\documentclass[12pt,a4paper]{amsart}
\usepackage{amsmath}
\usepackage{graphics}
\usepackage{style}
\usepackage{pgf}
\usepackage{graphics,graphicx,color,rotating}
\usepackage{boxedminipage,shadow}
\usepackage{marvosym}
\usepackage{hyperref}
\usepackage{movie15}
\usepackage{multirow}
\usepackage[latin1]{inputenc}
\usepackage{xspace}
\usepackage{epsfig} 
\usepackage{texdraw}
\usepackage{psfrag}
\usepackage{animate}
\usepackage{framed}
\usepackage{textcomp}

\usepackage{amsthm}

\usepackage[top=3cm, bottom=3cm, left=3cm, right=3cm]{geometry}

\def\cR{\boldsymbol{\mathcal{R}}}

\def\cS{\boldsymbol{\mathcal{S}}}

\def\cQ{\mathcal{Q}}

\def\eps{\varepsilon}
\def\0{{\bf 0}}

\def\x{\mathbf x}

\def\div{\mathrm{div}}

\def\T{\mathbb{T}}

\def\eps{\varepsilon}

\def\Di{\mathrm{D}}
\def\Re{\mathrm{Re}}
\def\We{\mathrm{We}}
\def\e{\mathrm{e}}

\newtheorem{theorem}{Theorem}
\newtheorem{lemma}{Lemma}
\newtheorem{corollary}[theorem]{Corollary}
\newtheorem{definition}{Definition}
\newtheorem{proposition}[lemma]{Proposition}
\newtheorem{remark}{Remark}
\newtheorem{notation}{Notation}

\begin{document}

\title{Viscoelastic flows in a rough channel:\\ a multiscale analysis}
\author{Laurent Chupin} 
\address{Universit\'e Blaise Pascal, Laboratoire de Math\'ematiques (CNRS UMR 6620), Campus des C\'ezeaux, 63177 Aubi\`ere cedex, France} 
\email{\texttt{laurent.chupin@math.univ-bpclermont.fr}}

\author{S\'ebastien Martin}
\address{Universit\'e Paris Descartes, Laboratoire MAP5 (CNRS UMR 8145), 45 rue des Saints-P\`eres, 75270 Paris cedex 06, France}
\email{\texttt{sebastien.martin@parisdescartes.fr}}

\date{\today}

\maketitle

\begin{abstract}
In this paper, we consider {\em viscoelastic} flows in a rough domain (with typical roughness patterns of size~$\eps \ll 1$). We present and rigorously justify an asymptotic expansion with respect to~$\eps$, at any order, based upon the definition of elementary problems: Oldroyd-type problems at the global scale defined on a smoothened domain and boundary-layer corrector problems. The resulting analysis guarantees optimality with respect to the truncation error.
\end{abstract}


\section{Introduction}\label{intro}

Many studies investigate the effect of wall roughness on Newtonian flows. In 1827, C. L. Navier \cite{Na27} was one of the first scientists to note that the roughness could drag a fluid. Since then, numerous studies attempted to prove mathematical results in this direction, see for instance the works of W.~J\"ager and A.~Mikelic~\cite{JM01}, Y.~Amirat and co-authors~\cite{ABLS01,AS96} and more recently the works of D.~Bresch and V.~Milisic~\cite{BMi10}. Note that all these works formulate the roughness using a periodic function (whose amplitude and period are supposed to be small). In a context of more general ``roughness'' patterns, there exists similar recent results, see~\cite{BaGV08,GV08}. All the previous work deal with a {\em Newtonian} flow, for which the Stokes or Navier-Stokes equations are classically considered.

Much literature research has been devoted to non-Newtonian fluids, in both mathematical aspects and applications. It is well known that numerous biological fluids, blood or physiological secretions like tears or synovial fluids, show these non-Newtonian characteristics. In engineering applications people are interested in controlling the flows characteristics to suit various requirements such as maintaining the fluid qualities in a wide range of temperatures and stresses. Introduction of additives lead to non-Newtonian behavior of the modern lubricants for instance. Another application domain is linked to polymers, whose non-Newtonian characteristics appear in a wide range of applications such as the molding or injection processes. Some particular classes of non-Newtonian models have often been considered. This includes the Bingham flow or the quasi-Newtonian fluids (Carreau's law \cite{CaKD79,Ke95}, the power law or Williamson's law, in which various stress-velocity relations are chosen \cite{TW98} or \cite{BlM95,BT95} for mathematical aspects) and also micropolar ones \cite{Lu99}. These models, however, consider the fluid as viscous and elasticity effects are neglected. The introduction of such a viscoelastic behavior is primarily described by the Weissenberg number, denoted~$\We$ which can be viewed as a measure of the elasticity of the fluid and is related to its characteristic relaxation time. One of the laws which seems the most able to describe viscoelastic flows is the Oldroyd model. This model is based on a constitutive equation which is an interpolation between purely viscous and purely elastic behaviors, thus introducing a supplementary parameter~$r$ which describes the relative proportion of both behaviors (the solvent to solute ratio). Considering the Oldroyd model \cite{Ol50}, the momentum, continuity and constitutive equations for an incompressible flow of such a non-Newtonian fluid are, respectively,
\begin{equation}\label{equation_base1}
\rho \bigg( \partial_{t} u + u \cdot \nabla u \bigg) - \eta (1-r) \Delta u + \nabla p - \div ~ \sigma = 0,
\end{equation}
\begin{equation}\label{equation_base2}
\div ~ u = 0,
\end{equation}
\begin{equation}\label{equation_base3}
\lambda \bigg( \partial_{t} \sigma + u \cdot \nabla \sigma + g_a(\nabla u,\sigma) \bigg) + \sigma - \delta \Delta \sigma = 2 \eta r \mathbb D(u). 
\end{equation}
\indent In these equations, $\rho$, $\eta$ and~$\lambda$ are positive constants which respectively correspond to the fluid density, the fluid viscosity and the relaxation time.
It is important to notice the presence of a term~$\delta \Delta \sigma$, $\delta >0$ corresponding to a spatial diffusion of the polymeric stresses. Usually ({\it i.e.} for the Oldroyd model) this term is deleted, but it can be physically justified: physical effects that can contribute to the diffusive process include hydrodynamic interactions \cite{EKL89}, particle diffusion \cite{BMAB93} and semiflexibility of polymer blends \cite{LF93}. Thus, the diffusive Oldroyd models have been the subject of intense studies related to the understanding of shear-banded flows or phase coexistence, see \cite{OL97,SYC96,OL99,LOB00,ORL00} and, as a consequence, the investigation of the mathematical properties of the diffusive model has gained an increasing interest recently~\cite{EKL89,BSS05,CK12,CM14}.

Equations \eqref{equation_base1}--\eqref{equation_base3} make up a system of~$10$ scalar equations with~$10$ unknowns: the lubricant velocity vector $u=(u_i)_{1\leq i\leq 3}$, the pressure~$p$ and the extra-stress symmetric tensor $\sigma=(\sigma_{i,j})_{1\leq i,j\leq 3}$.
The bilinear application~$g_a$, $-1\leq a \leq 1$, is defined by
$$g_a(\nabla u,\sigma) = \sigma \cdot \mathbb W(u) - \mathbb W(u) \cdot \sigma - a(\sigma \cdot \mathbb D(u) + \mathbb D(u) \cdot \sigma)$$
where~$\mathbb D(u)$ and~$\mathbb W(u)$ are respectively the symmetric and skew-symmetric parts of the velocity gradient~$\nabla u$. Usually, $\mathbb D(u)$ is called the rate of strain tensor and~$\mathbb W(u)$ is called the vorticity tensor. Notice that the parameter~$a$ is considered to interpolate between upper convected ($a=1$) and lower convective derivatives ($a=-1$), the case $a=0$ being the corotational case \cite{Jo98}. Note that taking $r=1$ allows us to recover various forms of the generalized Maxwell model. Conversely, a Newtonian flow is described by choosing $r=0$.

In this paper, we focus on {\em viscoelastic} flows in a rough domain (with typical roughness patterns of size~$\varepsilon \ll 1$). We present and rigorously justify an asymptotic expansion with respect to~$\varepsilon$. The development is done at any order, so that we are guaranteed to be optimal with respect to the truncation error. We also highlight the particular effects of roughness.

Several relevant questions are not addressed in this article. First, recent works on random roughness, see~\cite{BaGV08,GV08}, could make us think that our results can be extended to more general cases of roughness. In fact, the construction of our development strongly depends on the behavior of solutions of the Stokes equation on a half-space, whose lower boundary is periodic. The behavior of such solutions must be sufficiently decreasing at infinity to justify our development. Unfortunately, it seems that this decrease is only logarithmic in the case of a random boundary (while it is exponential in our periodic case). Second, another task related to the regularity of the roughness patterns is not addressed in this paper: what is the behavior of the solution when the patterns are not Lipschitz continuous? In particular, what is the influence of roughness jump discontinuities over the flow?
Finally, the choice to make appear a spatial diffusion ($\delta >0$) in the Oldroyd model could by argue. From a mathematical point of view, it is clearly an advantage since we know that such diffusion allows to have solutions to the initial problem~\eqref{equation_base1}--\eqref{equation_base3}.
Nevertheless, even if we admit the existence of a smooth solution to the initial problem without diffusion ($\delta = 0$), the development proposed with respect to the roughness parameter~$\varepsilon$ seems to be unsuitable.

The paper is composed of five sections. In Section~\ref{Section2}, we introduce the diffusive Oldroyd model, we precisely describe the roughness geometry and we recall a fundamental result: there exists a solution~$(u,p,\sigma)$ to the model. In Section~\ref{Section3} we introduce the ansatz, which is a formal asymptotic expansion of the solution~$(u,p,\sigma)$ with respect to the roughness parameter~$\varepsilon \ll 1$. By identifying the powers of~$\varepsilon$ in such a development, we obtain some elementary problems at any order. Section~\ref{Section4} is devoted to the mathematical study of these elementary problems: well-posedness and properties. Section~\ref{Section5} provides a rigorous justification of the asymptotic expansion by analyzing the remainder and deriving error estimates. In Section~\ref{Section6}, we show that it is possible to effectively determine all the contributions of the ansatz. We will notice that regarding each elementary problem and their overlaps, this crucial result is not obvious: we prove that the solution of each problem can be built using the only previous elementary solutions.

\section{The diffusive Oldroyd system in a rough channel: statement of the problem}\label{Section2}

As a matter of fact, the derivation of reduced models is crucial if one aims at performing numerical simulations of the flow. However, roughness patterns lead to a sharp increase of the computational costs because the mesh of the domain has to be built according to the constraints defined by the roughness patterns. In order to avoid such a costly procedure, reduced models can be defined by considering the fluid flow in the smooth domain (therefore avoiding heavy costs in terms of numerical computations by using coarse meshes) and adding a so-called boundary layer correction which takes into account the influence of the roughness pattern. More precisely, for a {\em Newtonian} flow, we can proceed as follows (see \cite{Na27,APV98,JM01,BMi10}, and also \cite{GV08,BaGV08} for random roughness patterns):
\begin{itemize}
\item {\em at order 0}: assume that~$u$ is the velocity fluid associated to the rough domain~$\omega_{\varepsilon}$ and~$u_0$ is the velocity fluid associated to the smooth domain (i.e. the domain has been truncated by considering a {\em flat} boundary instead of the oscillating one and imposing the Dirichlet condition on the smooth boundary), see Fig.~\ref{fig:roughnewtonian}. Then one has 
$$\| u - u_0 \|_{L^2(\omega_{\varepsilon})} \lesssim \varepsilon
\quad \text{and} \quad
\| u - u_0 \|_{H^1(\omega_{\varepsilon})} \lesssim \sqrt{\varepsilon}.$$
\item {\em at order 1}: in order to counterbalance the error introduced by the truncation of the domain, it is possible to introduce a corrector term which leads to the justification of wall laws. By considering the smooth domain with a slip condition defined as 
$$
\left\{\begin{aligned}
& u^{(1)} = \alpha_{11} \, \partial_{y} u^{(1)} + \alpha_{12} \, \partial_{y} u^{(2)} \\
& u^{(2)} = \alpha_{21} \, \partial_{y} u^{(1)} + \alpha_{22} \, \partial_{y} u^{(2)} \\
& u^{(3)} = 0
\end{aligned}
\right.
$$
on the smooth boundary (the reals~$\alpha_{i,j}$ being coefficients depending on the roughness shape), a velocity field~$\overline{u}$ is defined, leading to a refinement of the approximation.
For instance we have (see \cite{JM01})
$$\| u - \overline{u} \|_{L^2(\omega_{\varepsilon})} \lesssim \varepsilon^{3/2}
\quad \text{and} \quad
\| u - \overline{u} \|_{H^1(\omega_{\varepsilon})} \lesssim \varepsilon.$$
\end{itemize}
\begin{figure}[hbtp]
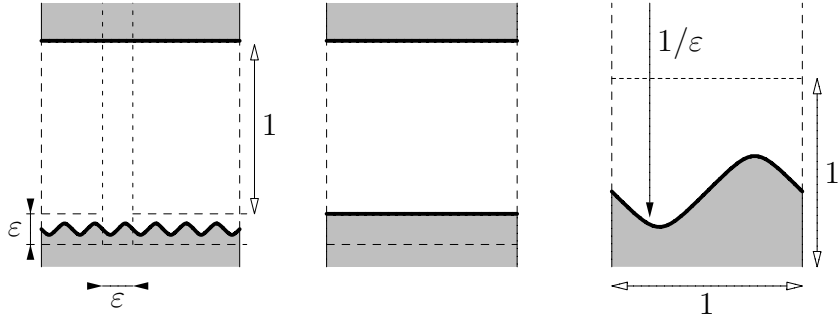

\centertexdraw{
\setgray 0
\drawdim cm \linewd 0.01
 \arrowheadsize l:0.20 w:0.10
\move(-5 3)\lvec(-5 2.5) \lvec(-2.4 2.5)\lvec(-2.4 3) \lfill f:0.75
\move(-5 -0.5) \lvec(-5 0) 
\clvec(-4.9 -0.1)(-4.9 -0.1)(-4.8 0) \clvec(-4.7 0.1)(-4.7 0.1)(-4.6 0) 
\clvec(-4.5 -0.1)(-4.5 -0.1)(-4.4 0) \clvec(-4.3 0.1)(-4.3 0.1)(-4.2 0) 
\clvec(-4.1 -0.1)(-4.1 -0.1)(-4.0 0) \clvec(-3.9 0.1)(-3.9 0.1)(-3.8 0) 
\clvec(-3.7 -0.1)(-3.7 -0.1)(-3.6 0) \clvec(-3.5 0.1)(-3.5 0.1)(-3.4 0) 
\clvec(-3.3 -0.1)(-3.3 -0.1)(-3.2 0) \clvec(-3.1 0.1)(-3.1 0.1)(-3.0 0) 
\clvec(-2.9 -0.1)(-2.9 -0.1)(-2.8 0) \clvec(-2.7 0.1)(-2.7 0.1)(-2.6 0) 
\clvec(-2.5 -0.1)(-2.5 -0.1)(-2.4 0) 
\lvec(-2.4 -0.5)
\lfill f:0.75
\move(-2.2 1.5) \arrowheadtype t:T \avec(-2.2 2.45)
\move(-2.2 1.5) \arrowheadtype t:T \avec(-2.2 0.20)
\move(-2.1 1.25) \textref h:L v:B \htext{{$1$}}
\drawdim cm \linewd 0.01
\lpatt(0.1 0.1) 
\move(-5 0) \lvec(-5 2.5) 
\move(-2.4 0) \lvec(-2.4 2.5) 
\move(-5.2 0.2) \lvec(-4.2 0.2) \move(-3.8 0.2) \lvec(-2.3 0.2)
\move(-5.2 -0.2) \lvec(-2.3 -0.2)
\drawdim cm \linewd 0.05
\lpatt(0.1 0) 
\move(-5 2.5) \lvec(-2.4 2.5)
\move(-5 0) 
\clvec(-4.9 -0.1)(-4.9 -0.1)(-4.8 0) \clvec(-4.7 0.1)(-4.7 0.1)(-4.6 0) 
\clvec(-4.5 -0.1)(-4.5 -0.1)(-4.4 0) \clvec(-4.3 0.1)(-4.3 0.1)(-4.2 0) 
\clvec(-4.1 -0.1)(-4.1 -0.1)(-4.0 0) \clvec(-3.9 0.1)(-3.9 0.1)(-3.8 0) 
\clvec(-3.7 -0.1)(-3.7 -0.1)(-3.6 0) \clvec(-3.5 0.1)(-3.5 0.1)(-3.4 0) 
\clvec(-3.3 -0.1)(-3.3 -0.1)(-3.2 0) \clvec(-3.1 0.1)(-3.1 0.1)(-3.0 0) 
\clvec(-2.9 -0.1)(-2.9 -0.1)(-2.8 0) \clvec(-2.7 0.1)(-2.7 0.1)(-2.6 0) 
\clvec(-2.5 -0.1)(-2.5 -0.1)(-2.4 0) 
\drawdim cm \linewd 0.01
\lpatt(0.1 0)
\move(-5.15 -0.25)\arrowheadtype t:F \avec(-5.15 -0.2)
\move(-5.15  0.25)\arrowheadtype t:F \avec(-5.15  0.20)
\move(-5.15  0.25)\lvec(-5.15 -0.25)
\move(-4.3 -0.75)\arrowheadtype t:F \avec(-4.2 -0.75)
\move(-3.7 -0.75)\arrowheadtype t:F \avec(-3.8 -0.75)
\move(-4.2 -0.75)\lvec(-3.8 -0.75)
\lpatt(0.05 0.1) 
\move(-4.2 -0.2)\lvec(-4.2 3)
\move(-3.8 -0.2)\lvec(-3.8 3)
\lpatt(0.1 0)
\move(-4 -0.85) \textref h:C v:T \htext{$\varepsilon$}
\move(-5.25 0) \textref h:R v:C \htext{$\varepsilon$}
\lpatt(0.1 0) 
\drawdim cm \linewd 0.01
\move(-1.25  3.0) \lvec(-1.25 2.5)\lvec(1.25 2.5)\lvec(1.25  3.0) \lfill f:0.75
\move(-1.25 -0.5) \lvec(-1.25 0.2)\lvec(1.25 0.2)\lvec(1.25 -0.5) \lfill f:0.75
\lpatt(0.1 0.1) 
\move(-1.25 0.2) \lvec(-1.25 2.5) 
\move( 1.25 0.2) \lvec( 1.25 2.5) 
\drawdim cm \linewd 0.05
\lpatt(0.1 0) 
\move(-1.25 2.5)\lvec(1.25 2.5)
\move(-1.25 0.2)\lvec(1.25 0.2)
\drawdim cm \linewd 0.01
\lpatt(0.1 0.1) 
\move(-1.25 -0.2)\lvec(1.25 -0.2)

\lpatt(0.1 0) 
\drawdim cm \linewd 0.01
\move(2.5 -0.5) \lvec( 2.5 0.5) \clvec(3.125 -0.125)(3.125 -0.125)(3.75 0.5) \clvec(4.375 1.125)(4.375 1.125)(5 0.5) \lvec(5 -0.5) \lfill f:0.75
\move(5.2  1.0) \arrowheadtype t:T \avec(5.2  2.0)
\move(5.2  1.0) \arrowheadtype t:T \avec(5.2 -0.5)
\move(5.3 0.75) \textref h:L v:C \htext{{$1$}}
\move(3.5 -0.75) \arrowheadtype t:T \avec(2.5 -0.75)
\move(3.5 -0.75) \arrowheadtype t:T \avec(5.0 -0.75)
\move(3.75 -0.85) \textref h:C v:T \htext{{$1$}}
\move(3.0 3) \arrowheadtype t:F \avec(3.0 0.15)
\move(3.1 2.25) \textref h:L v:B \htext{{${1/\varepsilon}$}}
\lpatt(0.1 0.1) 
\move(2.5 0.5) \lvec(2.5 3.0) 
\move(5.0 0.5) \lvec(5.0 3.0) 
\lpatt(0.05 0.05) 
\move(2.5 2) \lvec(5 2)
\drawdim cm \linewd 0.05
\lpatt(0.1 0) 
\move( 2.5 0.5) \clvec(3.125 -0.125)(3.125 -0.125)(3.75 0.5) \clvec(4.375 1.125)(4.375 1.125)(5 0.5)
}\caption{a) Rough domain.  b) {\em Flattened} domain: the rough boundary has been truncated by a smooth one. c) Boundary layer domain corresponding to a focus on the vicinity of the rough periodic boundary.}\label{fig:roughnewtonian}
\end{figure}

For a non-Newtonian flow of the viscoelastic type, we aim at describing the roughness effects {\em extensively} in the following sense:
\begin{itemize}
\item build an asymptotic expansion based upon elementary solutions (i.e. solutions of problems defined on smooth domains, thus avoiding complex geometries)~;
\item prove in a rigorous way that the asymptotic expansion is valid at any order~;
\item define an algorithm associated to an efficient numerical procedure.
\end{itemize}

Therefore we consider the Navier-Stokes equations with the Oldroyd model in a rough channel. Consider the domain
$$
\omega_{\varepsilon}:=\left\{ (x,y)\in \T^{d}\times \mathbb R,\ -\varepsilon H\left(\frac{x}{\varepsilon}\right)<y<1 \right\},
$$
where~$\T^d$ is the~$d$-dimensional torus, $d=1$ or $d=2$, and~$H$ a smooth periodic and positive function. The boundary~$\omega_\varepsilon$ is denoted~$\gamma_\varepsilon$ and it is composed of two connex components: The upper {\em smooth} boundary $\gamma^+ = \T^d\times \{1\}$ and the upper {\em highly oscillating} boundary which is denoted~$\gamma_\varepsilon^-$, see Fig.~\ref{fig:roughchannel}.

\begin{figure}[hbtp]
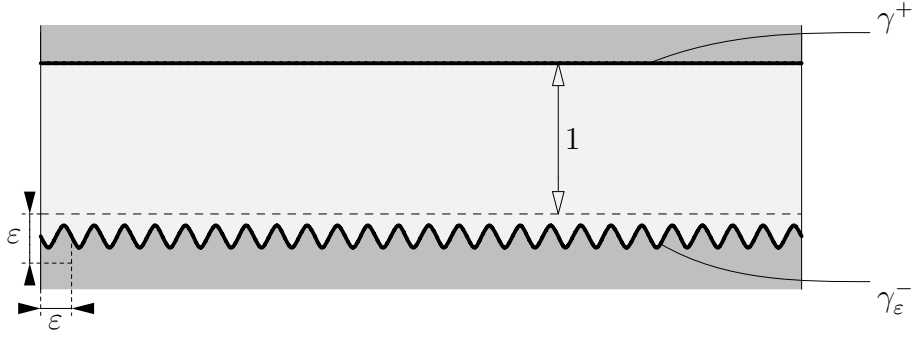

\centertexdraw{
\setgray 0
\drawdim cm \linewd 0.01
\arrowheadsize l:0.30 w:0.15
\move(-5 -0.5) \lvec(-5 0.2) 
\clvec(-4.9 0)(-4.9 0)(-4.8 0.2) \clvec(-4.7 0.4)(-4.7 0.4)(-4.6 0.2) 
\clvec(-4.5 0)(-4.5 0)(-4.4 0.2) \clvec(-4.3 0.4)(-4.3 0.4)(-4.2 0.2) 
\clvec(-4.1 0)(-4.1 0)(-4.0 0.2) \clvec(-3.9 0.4)(-3.9 0.4)(-3.8 0.2) 
\clvec(-3.7 0)(-3.7 0)(-3.6 0.2) \clvec(-3.5 0.4)(-3.5 0.4)(-3.4 0.2) 
\clvec(-3.3 0)(-3.3 0)(-3.2 0.2) \clvec(-3.1 0.4)(-3.1 0.4)(-3.0 0.2) 
\clvec(-2.9 0)(-2.9 0)(-2.8 0.2) \clvec(-2.7 0.4)(-2.7 0.4)(-2.6 0.2) 
\clvec(-2.5 0)(-2.5 0)(-2.4 0.2) \clvec(-2.3 0.4)(-2.3 0.4)(-2.2 0.2) 
\clvec(-2.1 0)(-2.1 0)(-2.0 0.2) \clvec(-1.9 0.4)(-1.9 0.4)(-1.8 0.2) 
\clvec(-1.7 0)(-1.7 0)(-1.6 0.2) \clvec(-1.5 0.4)(-1.5 0.4)(-1.4 0.2) 
\clvec(-1.3 0)(-1.3 0)(-1.2 0.2) \clvec(-1.1 0.4)(-1.1 0.4)(-1.0 0.2) 
\clvec(-0.9 0)(-0.9 0)(-0.8 0.2) \clvec(-0.7 0.4)(-0.7 0.4)(-0.6 0.2) 
\clvec(-0.5 0)(-0.5 0)(-0.4 0.2) \clvec(-0.3 0.4)(-0.3 0.4)(-0.2 0.2) 
\clvec(-0.1 0)(-0.1 0)(-0.0 0.2) \clvec( 0.1 0.4)( 0.1 0.4)( 0.2 0.2) 
\clvec( 0.3 0)( 0.3 0)( 0.4 0.2) \clvec( 0.5 0.4)( 0.5 0.4)( 0.6 0.2) 
\clvec( 0.7 0)( 0.7 0)( 0.8 0.2) \clvec( 0.9 0.4)( 0.9 0.4)( 1.0 0.2) 
\clvec( 1.1 0)( 1.1 0)( 1.2 0.2) \clvec( 1.3 0.4)( 1.3 0.4)( 1.4 0.2) 
\clvec( 1.5 0)( 1.5 0)( 1.6 0.2) \clvec( 1.7 0.4)( 1.7 0.4)( 1.8 0.2) 
\clvec( 1.9 0)( 1.9 0)( 2.0 0.2) \clvec( 2.1 0.4)( 2.1 0.4)( 2.2 0.2) 
\clvec( 2.3 0)( 2.3 0)( 2.4 0.2) \clvec( 2.5 0.4)( 2.5 0.4)( 2.6 0.2) 
\clvec( 2.7 0)( 2.7 0)( 2.8 0.2) \clvec( 2.9 0.4)( 2.9 0.4)( 3.0 0.2) 
\clvec( 3.1 0)( 3.1 0)( 3.2 0.2) \clvec( 3.3 0.4)( 3.3 0.4)( 3.4 0.2) 
\clvec( 3.5 0)( 3.5 0)( 3.6 0.2) \clvec( 3.7 0.4)( 3.7 0.4)( 3.8 0.2) 
\clvec( 3.9 0)( 3.9 0)( 4.0 0.2) \clvec( 4.1 0.4)( 4.1 0.4)( 4.2 0.2) 
\clvec( 4.3 0)( 4.3 0)( 4.4 0.2) \clvec( 4.5 0.4)( 4.5 0.4)( 4.6 0.2) 
\clvec( 4.7 0)( 4.7 0)( 4.8 0.2) \clvec( 4.9 0.4)( 4.9 0.4)( 5.0 0.2) 
\lvec(5 2.9)\lvec(-5 2.9) \lvec(-5 0.2)
\lfill f:0.95
\move(-5 3) \lvec(-5 2.5) \lvec(5 2.5) \lvec(5 3)
\lfill f:0.75
\move(-5 -0.5) \lvec(-5 0.2) 
\clvec(-4.9 0)(-4.9 0)(-4.8 0.2) \clvec(-4.7 0.4)(-4.7 0.4)(-4.6 0.2) 
\clvec(-4.5 0)(-4.5 0)(-4.4 0.2) \clvec(-4.3 0.4)(-4.3 0.4)(-4.2 0.2) 
\clvec(-4.1 0)(-4.1 0)(-4.0 0.2) \clvec(-3.9 0.4)(-3.9 0.4)(-3.8 0.2) 
\clvec(-3.7 0)(-3.7 0)(-3.6 0.2) \clvec(-3.5 0.4)(-3.5 0.4)(-3.4 0.2) 
\clvec(-3.3 0)(-3.3 0)(-3.2 0.2) \clvec(-3.1 0.4)(-3.1 0.4)(-3.0 0.2) 
\clvec(-2.9 0)(-2.9 0)(-2.8 0.2) \clvec(-2.7 0.4)(-2.7 0.4)(-2.6 0.2) 
\clvec(-2.5 0)(-2.5 0)(-2.4 0.2) \clvec(-2.3 0.4)(-2.3 0.4)(-2.2 0.2) 
\clvec(-2.1 0)(-2.1 0)(-2.0 0.2) \clvec(-1.9 0.4)(-1.9 0.4)(-1.8 0.2) 
\clvec(-1.7 0)(-1.7 0)(-1.6 0.2) \clvec(-1.5 0.4)(-1.5 0.4)(-1.4 0.2) 
\clvec(-1.3 0)(-1.3 0)(-1.2 0.2) \clvec(-1.1 0.4)(-1.1 0.4)(-1.0 0.2) 
\clvec(-0.9 0)(-0.9 0)(-0.8 0.2) \clvec(-0.7 0.4)(-0.7 0.4)(-0.6 0.2) 
\clvec(-0.5 0)(-0.5 0)(-0.4 0.2) \clvec(-0.3 0.4)(-0.3 0.4)(-0.2 0.2) 
\clvec(-0.1 0)(-0.1 0)(-0.0 0.2) \clvec( 0.1 0.4)( 0.1 0.4)( 0.2 0.2) 
\clvec( 0.3 0)( 0.3 0)( 0.4 0.2) \clvec( 0.5 0.4)( 0.5 0.4)( 0.6 0.2) 
\clvec( 0.7 0)( 0.7 0)( 0.8 0.2) \clvec( 0.9 0.4)( 0.9 0.4)( 1.0 0.2) 
\clvec( 1.1 0)( 1.1 0)( 1.2 0.2) \clvec( 1.3 0.4)( 1.3 0.4)( 1.4 0.2) 
\clvec( 1.5 0)( 1.5 0)( 1.6 0.2) \clvec( 1.7 0.4)( 1.7 0.4)( 1.8 0.2) 
\clvec( 1.9 0)( 1.9 0)( 2.0 0.2) \clvec( 2.1 0.4)( 2.1 0.4)( 2.2 0.2) 
\clvec( 2.3 0)( 2.3 0)( 2.4 0.2) \clvec( 2.5 0.4)( 2.5 0.4)( 2.6 0.2) 
\clvec( 2.7 0)( 2.7 0)( 2.8 0.2) \clvec( 2.9 0.4)( 2.9 0.4)( 3.0 0.2) 
\clvec( 3.1 0)( 3.1 0)( 3.2 0.2) \clvec( 3.3 0.4)( 3.3 0.4)( 3.4 0.2) 
\clvec( 3.5 0)( 3.5 0)( 3.6 0.2) \clvec( 3.7 0.4)( 3.7 0.4)( 3.8 0.2) 
\clvec( 3.9 0)( 3.9 0)( 4.0 0.2) \clvec( 4.1 0.4)( 4.1 0.4)( 4.2 0.2) 
\clvec( 4.3 0)( 4.3 0)( 4.4 0.2) \clvec( 4.5 0.4)( 4.5 0.4)( 4.6 0.2) 
\clvec( 4.7 0)( 4.7 0)( 4.8 0.2) \clvec( 4.9 0.4)( 4.9 0.4)( 5.0 0.2) 
\lvec(5 -0.5)
\lfill f:0.75
\move(1.8 1)\arrowheadtype t:T \avec(1.8 2.5)
\move(1.8 1)\arrowheadtype t:T \avec(1.8 0.5)
%
\move(1.9 1.5) \textref h:L v:C \htext{$1$}
\move(-4.8 -0.85) \textref h:C v:T \htext{$\varepsilon$}
\move(-5.25 0.2) \textref h:R v:C \htext{$\varepsilon$}

\move(6 2.9) \textref h:L v:B \htext{$\gamma^+$} \move(5.9 2.9) \clvec(5 2.9)(4 2.9)( 3 2.5)
\move(6 -0.4) \textref h:L v:T \htext{$\gamma^-_{\varepsilon}$} \move(5.9 -0.4) \clvec(5 -0.4)(4 -0.4)( 3.15 0.1)

\move(-5 0.2) 
\clvec(-4.9 0)(-4.9 0)(-4.8 0.2) \clvec(-4.7 0.4)(-4.7 0.4)(-4.6 0.2) 
\clvec(-4.5 0)(-4.5 0)(-4.4 0.2) \clvec(-4.3 0.4)(-4.3 0.4)(-4.2 0.2) 
\clvec(-4.1 0)(-4.1 0)(-4.0 0.2) \clvec(-3.9 0.4)(-3.9 0.4)(-3.8 0.2) 
\clvec(-3.7 0)(-3.7 0)(-3.6 0.2) \clvec(-3.5 0.4)(-3.5 0.4)(-3.4 0.2) 
\clvec(-3.3 0)(-3.3 0)(-3.2 0.2) \clvec(-3.1 0.4)(-3.1 0.4)(-3.0 0.2) 
\clvec(-2.9 0)(-2.9 0)(-2.8 0.2) \clvec(-2.7 0.4)(-2.7 0.4)(-2.6 0.2) 
\clvec(-2.5 0)(-2.5 0)(-2.4 0.2) \clvec(-2.3 0.4)(-2.3 0.4)(-2.2 0.2) 
\clvec(-2.1 0)(-2.1 0)(-2.0 0.2) \clvec(-1.9 0.4)(-1.9 0.4)(-1.8 0.2) 
\clvec(-1.7 0)(-1.7 0)(-1.6 0.2) \clvec(-1.5 0.4)(-1.5 0.4)(-1.4 0.2) 
\clvec(-1.3 0)(-1.3 0)(-1.2 0.2) \clvec(-1.1 0.4)(-1.1 0.4)(-1.0 0.2) 
\clvec(-0.9 0)(-0.9 0)(-0.8 0.2) \clvec(-0.7 0.4)(-0.7 0.4)(-0.6 0.2) 
\clvec(-0.5 0)(-0.5 0)(-0.4 0.2) \clvec(-0.3 0.4)(-0.3 0.4)(-0.2 0.2) 
\clvec(-0.1 0)(-0.1 0)(-0.0 0.2) \clvec( 0.1 0.4)( 0.1 0.4)( 0.2 0.2) 
\clvec( 0.3 0)( 0.3 0)( 0.4 0.2) \clvec( 0.5 0.4)( 0.5 0.4)( 0.6 0.2) 
\clvec( 0.7 0)( 0.7 0)( 0.8 0.2) \clvec( 0.9 0.4)( 0.9 0.4)( 1.0 0.2) 
\clvec( 1.1 0)( 1.1 0)( 1.2 0.2) \clvec( 1.3 0.4)( 1.3 0.4)( 1.4 0.2) 
\clvec( 1.5 0)( 1.5 0)( 1.6 0.2) \clvec( 1.7 0.4)( 1.7 0.4)( 1.8 0.2) 
\clvec( 1.9 0)( 1.9 0)( 2.0 0.2) \clvec( 2.1 0.4)( 2.1 0.4)( 2.2 0.2) 
\clvec( 2.3 0)( 2.3 0)( 2.4 0.2) \clvec( 2.5 0.4)( 2.5 0.4)( 2.6 0.2) 
\clvec( 2.7 0)( 2.7 0)( 2.8 0.2) \clvec( 2.9 0.4)( 2.9 0.4)( 3.0 0.2) 
\clvec( 3.1 0)( 3.1 0)( 3.2 0.2) \clvec( 3.3 0.4)( 3.3 0.4)( 3.4 0.2) 
\clvec( 3.5 0)( 3.5 0)( 3.6 0.2) \clvec( 3.7 0.4)( 3.7 0.4)( 3.8 0.2) 
\clvec( 3.9 0)( 3.9 0)( 4.0 0.2) \clvec( 4.1 0.4)( 4.1 0.4)( 4.2 0.2) 
\clvec( 4.3 0)( 4.3 0)( 4.4 0.2) \clvec( 4.5 0.4)( 4.5 0.4)( 4.6 0.2) 
\clvec( 4.7 0)( 4.7 0)( 4.8 0.2) \clvec( 4.9 0.4)( 4.9 0.4)( 5.0 0.2) 
\lpatt(0.1 0.1) \move(-5 0.5) \lvec(5 0.5)
\lpatt(0.05 0.05) 
\move(-5.0   0.50) \lvec(-5.25  0.50) 
\move(-4.6  -0.15) \lvec(-5.25 -0.15)
\move(-5.0 -0.5) \lvec(-5.0 -0.85) 
\move(-4.6  0.2) \lvec(-4.6 -0.85)
\lpatt(0.1 0)
\move(-5.15 -0.25)\arrowheadtype t:F \avec(-5.15 -0.15)
\move(-5.15  0.60)\arrowheadtype t:F \avec(-5.15  0.50)
\move(-5.15  0.50)\lvec(-5.15 -0.15)
\move(-5.10 -0.75)\arrowheadtype t:F \avec(-5.00 -0.75)
\move(-4.50 -0.75)\arrowheadtype t:F \avec(-4.60 -0.75)
\move(-5.00 -0.75)\lvec(-4.60 -0.75)
\drawdim cm \linewd 0.05
\lpatt(0.1 0) 
\move(-5 0.2) 
\clvec(-4.9 0)(-4.9 0)(-4.8 0.2) \clvec(-4.7 0.4)(-4.7 0.4)(-4.6 0.2) 
\clvec(-4.5 0)(-4.5 0)(-4.4 0.2) \clvec(-4.3 0.4)(-4.3 0.4)(-4.2 0.2) 
\clvec(-4.1 0)(-4.1 0)(-4.0 0.2) \clvec(-3.9 0.4)(-3.9 0.4)(-3.8 0.2) 
\clvec(-3.7 0)(-3.7 0)(-3.6 0.2) \clvec(-3.5 0.4)(-3.5 0.4)(-3.4 0.2) 
\clvec(-3.3 0)(-3.3 0)(-3.2 0.2) \clvec(-3.1 0.4)(-3.1 0.4)(-3.0 0.2) 
\clvec(-2.9 0)(-2.9 0)(-2.8 0.2) \clvec(-2.7 0.4)(-2.7 0.4)(-2.6 0.2) 
\clvec(-2.5 0)(-2.5 0)(-2.4 0.2) \clvec(-2.3 0.4)(-2.3 0.4)(-2.2 0.2) 
\clvec(-2.1 0)(-2.1 0)(-2.0 0.2) \clvec(-1.9 0.4)(-1.9 0.4)(-1.8 0.2) 
\clvec(-1.7 0)(-1.7 0)(-1.6 0.2) \clvec(-1.5 0.4)(-1.5 0.4)(-1.4 0.2) 
\clvec(-1.3 0)(-1.3 0)(-1.2 0.2) \clvec(-1.1 0.4)(-1.1 0.4)(-1.0 0.2) 
\clvec(-0.9 0)(-0.9 0)(-0.8 0.2) \clvec(-0.7 0.4)(-0.7 0.4)(-0.6 0.2) 
\clvec(-0.5 0)(-0.5 0)(-0.4 0.2) \clvec(-0.3 0.4)(-0.3 0.4)(-0.2 0.2) 
\clvec(-0.1 0)(-0.1 0)(-0.0 0.2) \clvec( 0.1 0.4)( 0.1 0.4)( 0.2 0.2) 
\clvec( 0.3 0)( 0.3 0)( 0.4 0.2) \clvec( 0.5 0.4)( 0.5 0.4)( 0.6 0.2) 
\clvec( 0.7 0)( 0.7 0)( 0.8 0.2) \clvec( 0.9 0.4)( 0.9 0.4)( 1.0 0.2) 
\clvec( 1.1 0)( 1.1 0)( 1.2 0.2) \clvec( 1.3 0.4)( 1.3 0.4)( 1.4 0.2) 
\clvec( 1.5 0)( 1.5 0)( 1.6 0.2) \clvec( 1.7 0.4)( 1.7 0.4)( 1.8 0.2) 
\clvec( 1.9 0)( 1.9 0)( 2.0 0.2) \clvec( 2.1 0.4)( 2.1 0.4)( 2.2 0.2) 
\clvec( 2.3 0)( 2.3 0)( 2.4 0.2) \clvec( 2.5 0.4)( 2.5 0.4)( 2.6 0.2) 
\clvec( 2.7 0)( 2.7 0)( 2.8 0.2) \clvec( 2.9 0.4)( 2.9 0.4)( 3.0 0.2) 
\clvec( 3.1 0)( 3.1 0)( 3.2 0.2) \clvec( 3.3 0.4)( 3.3 0.4)( 3.4 0.2) 
\clvec( 3.5 0)( 3.5 0)( 3.6 0.2) \clvec( 3.7 0.4)( 3.7 0.4)( 3.8 0.2) 
\clvec( 3.9 0)( 3.9 0)( 4.0 0.2) \clvec( 4.1 0.4)( 4.1 0.4)( 4.2 0.2) 
\clvec( 4.3 0)( 4.3 0)( 4.4 0.2) \clvec( 4.5 0.4)( 4.5 0.4)( 4.6 0.2) 
\clvec( 4.7 0)( 4.7 0)( 4.8 0.2) \clvec( 4.9 0.4)( 4.9 0.4)( 5.0 0.2) 
\move(-5 2.5) \lvec(5 2.5)
}\caption{Channel with oscillating boundary.}\label{fig:roughchannel}
\end{figure}

Then we consider the following set of equations:
\begin{equation}\label{problem0}
\left\{
\begin{array}{rcll}
\Re \, ( u \cdot \nabla u ) - (1-r)\Delta u + \nabla p &=& \div \, \sigma + f, \qquad & \textnormal{in~$\omega_\varepsilon$,}\\
\div\, u &=& 0,& \textnormal{in~$\omega_\varepsilon$,}\\
\We \, ( u \cdot \nabla \sigma + g_{a}(\nabla u , \sigma) ) + \sigma - \Di \, \Delta \sigma &=& 2r \D(u),& \textnormal{in~$\omega_\varepsilon$,}\\
u &=& 0,& \textnormal{on~$\gamma_\varepsilon$,}\\
\Di \, \partial_n \sigma &=& 0,& \textnormal{on~$\gamma_\varepsilon$.}
\end{array}
\right.
\end{equation}
This system is the non-dimensional version of the system~\eqref{equation_base1}--\eqref{equation_base3}. We have introduced the Reynolds number~$\Re$, the Weissenberg number~$\We$, a relaxation parameter~$r\in [0,1]$ and the diffusion coefficient~$\Di$.

In the next sections, we aim at describing the structure of the solution~$(u,p)$ with respect to the roughness number~$\varepsilon$. Before entering into this description, let us recall the main mathematical results related to the stationary diffusive Oldroyd model. The problem defined in a strong form can be associated to a variational formulation. Then we have (see \cite{CM14} for details):

\begin{theorem}\label{theorem}
Let $f\in H^{-1}(\omega_{\varepsilon})^3$. Let $\Re \geq 0$, $\We\geq 0$, $0< r<1$, $-1\leq a\leq 1$ and~$\Di>0$. Let us introduce the following constants:
$$
C_{(\mathrm{I})} := \frac{8|a| C_{\omega_\varepsilon}^2 \We \|f\|_{H^{-1}}}{\min(1-r,\Di)^2},
\quad
C_{(\mathrm{II})} := \frac{\sqrt{2r}\min(1-r,\Di)}{4|a| C_{\omega_\varepsilon}^2 \We} \left( 1 - \sqrt{1 - C_{(\mathrm{I})}} \right),
$$
where~$C_{\omega_{\varepsilon}}$ is a constant which only depends on the domain~$\omega_{\varepsilon}$.

\begin{itemize}
\item {\em Existence.} If $C_{(\mathrm{I})} \leq 1$ then, the problem~\eqref{problem0} admits a variational solution $(u, \sigma)$ which satisfies
$$2r \|\nabla u\|_{L^2}^2 + \|\sigma\|_{H^1}^2 \leq C_{(\mathrm{II})}^2.$$
Moreover there exists a pressure field~$p\in L^2(\Omega)$ such that~$(u,p)$ satisfies the first equation of the problem~\eqref{problem0} in the sense of distributions.
\item {\em Uniqueness.} The variational formulation of the problem~\eqref{problem0} admits at most one solution if {\em one} of the following conditions is satisfied:
\begin{enumerate}
\item[a)]~$\|f\|_{H^{-1}}$ is small enough;
\item[b)]~$\Re$ and~$\We$ are small enough.
\end{enumerate}
\item {\em Regularity.} If~$f$ is regular then the variational solution is regular and satisfies the problem~\eqref{problem0} in a classical sense.
\end{itemize}
\end{theorem}

Let us remark that the corotationnal case (namely $a=0$) allows us to get rid of the smallness assumption on the data.

\section{Asymptotic expansion}\label{Section3}

\subsection{Main ideas: ansatz}\label{sec:2:1}

Let us describe the structure of the solution by using a suitable {\em ansatz}: 
$$
\left\{
\begin{aligned}
u(x,y) & = u_0(x,y) + \varepsilon U_1\Big(x,\frac{x}{\varepsilon},\frac{y}{\varepsilon}\Big) + \varepsilon u_1(x,y) +\varepsilon^2 U_2\Big(x,\frac{x}{\varepsilon},\frac{y}{\varepsilon}\Big) + \cdots \\
& = \displaystyle\sum_{k=0}^{+\infty} \varepsilon^k \left( u_k(x,y) + U_k\Big(x,\frac{x}{\varepsilon},\frac{y}{\varepsilon}\Big) \right), \\
p(x,y) & = p_0(x,y) + P_1\Big(x,\frac{x}{\varepsilon},\frac{y}{\varepsilon}\Big) + \varepsilon p_1(x,y) + \varepsilon P_2\Big(x,\frac{x}{\varepsilon},\frac{y}{\varepsilon}\Big) +\cdots\\ 
& = \displaystyle\sum_{k=0}^{+\infty} \varepsilon^k \left( p_k(x,y) + P_{k+1}\Big(x,\frac{x}{\varepsilon},\frac{y}{\varepsilon}\Big) \right), \\
\sigma(x,y) & = \sigma_0(x,y) + \varepsilon \Sigma_1\Big(x,\frac{x}{\varepsilon},\frac{y}{\varepsilon}\Big) + \varepsilon \sigma_1(x,y) + \varepsilon^2 \Sigma_2\Big(x,\frac{x}{\varepsilon},\frac{y}{\varepsilon}\Big) + \cdots\\ 
& = \displaystyle \sum_{k=0}^{+\infty} \varepsilon^k \left( \sigma_k(x,y) + \Sigma_k\Big(x,\frac{x}{\varepsilon},\frac{y}{\varepsilon}\Big) \right).
\end{aligned}
\right.
$$
The definition of the asymptotic expansion has to be completed by the description of the problems satisfied by the {\em elementary solutions}. Then another task consists in showing that each elementary problem is well posed. The final task consists in showing that the level of truncation in the asymptotic expansion is directly related to the quality of the approximation of the exact solution. 

In order to identify the {\em elementary problems} satisfied by the {\em elementary solutions}, we proceed as follows:
\begin{enumerate}
\item separation of the macroscopic variables $(x,y)$ and the microscopic ones which will be denoted $(X,Y)=(\frac{x}{\varepsilon},\frac{y}{\varepsilon})$~; 
\item identification of terms with the same order with respect to~$\varepsilon$ in the equations~;
\item identification of terms with the same order with respect to~$\varepsilon$ in the boundary conditions.
\end{enumerate}

\begin{notation}
We will use lowercase letters to denote elements corresponding to the real physical domain ($x$, $u$, $\sigma$, $\omega$...) and uppercase letters for all that concerns the microscopic field ($X$, $U$, $\Sigma$, $\Omega$...).
\end{notation}

To complete this subsection, and before getting into the details, let us explain the way the asymptotic expansion has been built: at main order, the fluid flow is governed by a classical viscoelastic model with boundary conditions located at the flat bottom $y=0$ (definition of $(u_0,p_0,\sigma_0)$). But, of course, the boundary layer has been omitted and, in fact, the boundary condition should have been imposed on the oscillating boundary instead of the smooth one~; it can be shown that the resulting error is of order~$\varepsilon$ (the size of the boundary layer) and, therefore, a so-called {\em boundary layer correction} is introduced in order to counterbalance the mentioned boundary value default (definition of $(U_{1},P_{1},\Sigma_{1})$). Now if we analyze the approximate solution defined as $(u_0+\varepsilon U_1,p_0+ P_1,\sigma_0+\varepsilon \Sigma_1)$, the equations in the domain and the boundary condition at the oscillating boundary are satisfied by means of construction~; unfortunately, the boundary condition on the upper boundary is not satisfied by the approximate solution because of the behavior of the boundary layer corrector at infinity. However the error is of order~$\varepsilon$ and it is is possible to build a solution on the smooth domain (definition of $(u_{1},p_{1},\sigma_{1})$) which counterbalances this boundary value default on the upper boundary. Again, at his step, by considering the smooth domain only (in particular, homogeneous boundary conditions are considered at the lower {\em smooth} boundary), the boundary layer has been omitted and, in fact, the boundary condition should have been imposed on the oscillating boundary~; the resulting error is {\em now} of order~$\varepsilon^2$. Thus, if we compare $(u_\varepsilon,p_\varepsilon,\sigma_\varepsilon)-(u_0,p_0,\sigma_0)$ and $(u_\varepsilon,p_\varepsilon,\sigma_\varepsilon)-(u_0+\varepsilon U_1+\varepsilon u_1,p_0+ P_1+\varepsilon p_1,\sigma_0+\varepsilon \Sigma_1+\varepsilon \sigma_1)$, the error has been decreased by an order of magnitude and, besides, the same procedure can be applied by introducing a boundary layer correction which counterbalances the boundary value default of order~$\varepsilon^2$ at the oscillating boundary.

In a more general way, the boundary value default introduced at the oscillating boundary can be counterbalanced by a boundary layer correction~; the resulting boundary value default at the upper boundary can be counterbalanced by a viscoelastic flow defined in the {\em smooth} domain. Through this procedure, the resulting approximation satisfies the equations in the domain, the boundary condition at the upper boundary and the error on the boundary condition on the oscillating boundary has been decreased by an order of magnitude.

Let us define the two rescaled sub-domains. As a matter of fact, the main flow is defined on the smooth domain whereas, due to the consideration of the roughness patterns, the boundary layer is rescaled by the homothetic transformation $(X,Y):=(\frac{x}{\varepsilon},\frac{y}{\varepsilon})$.

\begin{definition}
The {\em smooth} domain is defined by 
$$
\omega_0:=\big\{(x,y) \in \T^d \times \R, \quad 0 < y < 1 \big\}.
$$
The upper boundary equals~$\gamma^+$ and the lower boundary corresponds to~$\gamma^-_0$. The normal outward unit vector the lower (resp. upper) boundary is $n=(0,-1)$ (resp. $(0,1)$) on~$\gamma^-_0$ (resp. $\gamma^+$).
\end{definition}

\begin{definition}
The {\em boundary layer} domain is defined by 
$$
\Omega=\big\{(X,Y)\in \T^d\times \R, \quad -H(X) < Y \big\}.
$$
The boundary $\{(X,Y) \in \T^d \times \R, \quad Y = -H(X) \}$ is denoted~$\Gamma$. We denote by $N:= -(\nabla H, 1)$ the outward vector to the lower boundary. Note that~$N$ is not a unit vector.
\end{definition}

\begin{notation}
The usual notation for classical operator of derivation are~$\nabla$, $\div$ and~$\Delta$.
The problems considered in this paper make appear two kinds of functions: the first ones, like the velocity~$u$, which only depend on the macroscopic variables $(x,y)$, and the others, like the velocity~$U$ which depend on $(x,X,Y)$.
In the first case the classical operators are defined as usual, for instance
$$
\Delta u = \sum_{\ell=1}^d \partial_{x_\ell}^2 u + \partial_{y}^2 u.
$$
In the case of function depending on $(x,X,Y)$, the notations are the following:
$$
\Delta U = \sum_{\ell=1}^d \partial_{X_\ell}^2 U + \partial_Y^2 U
\quad \text{and} \quad
\Delta_x U = \sum_{\ell=1}^d \partial_{x_\ell}^2 U.
$$
Notice that the divergence with respect to the variable~$x$ is also defined for a function $U:\R^{d+1} \to \R^{d+1}$ by
$$
\div_x U = \sum_{\ell=1}^d \partial_{x_\ell} U^{(\ell)}.
$$
The normal derivative of a function~$U$ defined on~$\Gamma$ is defined by 
$$
\partial_N U = (N \cdot \nabla) U = - \sum_{\ell=1}^d \partial_{x_\ell} H \, \partial_{x_\ell} U - \partial_Y U.
$$
\end{notation}

\subsection{Elementary problems}

Let us define the problem at the main scale: it consists in considering the viscoelastic problem on the smooth domain, i.e. by truncating the rough boundary from the initial domain, associated to homogeneous conditions on both boundaries.\medskip

\noindent $\blacksquare$ {\bf Main order:}
$$
\mathrm{pb}^{(0)}\left\{
\begin{array}{rcll}
\Re \, ( u_0\cdot \nabla u_0 ) - (1-r)\Delta u_0 + \nabla p_0 - \div(\sigma_0) &=& f_0 + a_0,& \text{on~$\omega_0$}\\[0.2cm]
\div (u_0) &=& 0, & \text{on~$\omega_0$}\\[0.2cm]
\We \, ( u_0 \cdot \nabla \sigma_0 + g_a(\nabla u_0, \sigma_0) ) + \sigma_0 - \Di \Delta \sigma_0 &=& 2r \D(u_0), & \text{on~$\omega_0$}\\[0.2cm]
u_0&=&0, & \text{on~$\gamma^-_0\cup\gamma^+$}\\[0.2cm]
\partial_{n}\sigma_0&=&0, & \text{on~$\gamma^+$}\\[0.2cm]
\partial_{n}\sigma_0&=&b_0, & \text{on~$\gamma^-_0$}\\[0.2cm] 
\end{array}
\right.
$$
where $f_0=f$ is the source term coming from the modelisation, and where the constants~$a_{0}$ and~$b_{0}$ have to be fixed.\medskip

\begin{remark} The choice of constants~$a_{0}$ and~$b_{0}$ will be discussed further. Roughly speaking, they play the role of a degree of freedom which will be fixed in order to ensure the well-posedness of elementary problems and a suitable behavior of elementary solutions to be defined. From a practical point of view, we will set $a_{0}=b_{0}=0$.
\end{remark}

Imposing the homogeneous Dirichlet condition on the smooth lower boundary~$\gamma^-_0$ instead of the oscillating one~$\gamma^-_\varepsilon$ is a source of error. Indeed, $(u_0,p_0,\sigma_0)$ satisfies all the equations except the boundary condition at the oscillating boundary~$\gamma^-_\varepsilon$. The value at the oscillating boundary can be determined by a Taylor expansion:
$$\begin{array}{rcl}
u_{0}(x,-\varepsilon H(\frac x \varepsilon))&=& \displaystyle\sum_{k=0}^{+\infty} \displaystyle\frac{\left(-\varepsilon H(\frac x \varepsilon)\right)^k}{k!} \partial_{y}^k u_{0}(x,0).
\end{array}$$
Since $u_{0}(x,0)=0$, one can check that $u_{0}|_{\gamma_{\varepsilon}^-}$ is of order~$\varepsilon$, namely
$$\begin{array}{rcl}
u_{0}(x,-\varepsilon H(\frac x \varepsilon)) & = & -\varepsilon H(\frac x \varepsilon) \partial_{y} u_{0}(x,0) + \mathcal O(\varepsilon^2),
\end{array}$$
and this leading term of order~$\varepsilon$ will be counterbalanced by first the boundary layer corrector. 

\begin{remark} Note that the other terms (of orders~$\varepsilon^2$, $\varepsilon^3$ etc.) will be treated and counterbalanced in subsequent boundary layer problems.
\end{remark}

\begin{remark} The same methodology applies for the Neumann condition related to the elastic tensor.
\end{remark}

As a consequence of the previous remarks, we define the following boundary layer problem.\medskip

\noindent $\blacksquare$ {\bf Correction with boundary layer n$^{\mathrm o}$ 1:}
$$
\mathrm{PB}^{(1)}\left\{
\begin{array}{rclr}
 - (1-r)\Delta U_1 + \nabla P_1 &=& 0, & \text{on~$\Omega$}\\[0.2cm]
 \div (U_1) &=& 0, & \text{on~$\Omega$}\\[0.2cm]
 U_1 &=& H \partial_y u_0|_{\gamma^-_\varepsilon}, & \text{on~$\Gamma$}\\[0.4cm]
-\Di\, \Delta \Sigma_1 &=& 0, & \text{on~$\Omega$}\\[0.2cm]
\partial_N \Sigma_1 &=& (\nabla H \cdot \nabla_x) \sigma_0|_{\gamma^-_\varepsilon} - b_0 & \text{on~$\Gamma$}.
\end{array}
\right.
$$

\begin{remark} Let us present the behavior of the solution (the proofs will be given by proposition~\ref{prop:blp1} later).
\begin{enumerate}
\item When $Y\rightarrow +\infty$, the velocity~$U_1$ exponentially decreases towards the constant defined as
$$U_\infty:=\displaystyle \lim_{Y\to \infty} \int_{\T^d} U_1(\cdot,Y).$$
\item For $b_0=0$, the elastic tensor~$\Sigma_1$ exponentially decreases towards~$0$.
\end{enumerate}
\end{remark}

At this stage, it can be shown that the approximation satisfies the equations in the domain and at the oscillating boundary. At the upper boundary, the Neumann boundary condition for the elastic tensor is satisfied because of the exponential decay of~$\sigma_0$ at infinity. However the homogeneous Dirichlet condition is not satisfied for the velocity field. This is why a so-called main corrector defined on the smooth domain is defined in order to counterbalance the boundary default introduced by the approximation at the upper boundary.

Defining the linear operators~$\mathcal L^{(A)}$, $\mathcal L^{(B)}$ and~$\mathcal L^{(C)}$ by 
$$\begin{array}{lcl}
\mathcal L^{(A)}(u_1)&:=& \Re \, ( u_1\cdot \nabla u_0 + u_0\cdot \nabla u_1 ),\\
\mathcal L^{(B)}(u_1)&:=& \We \,(u_1 \cdot \nabla \sigma_0 + g_a(\nabla u_1 , \sigma_0)),\\
\mathcal L^{(C)}(\sigma_1)&:=&\We \,(u_0 \cdot \nabla \sigma_1 + g_a(\nabla u_0 , \sigma_1)),
\end{array}$$
we define the following problems:\medskip

\noindent $\blacksquare$ {\bf Main correction n$^{\mathrm o}$ 1:} 
$$
\mathrm{pb}^{(1)}\left\{
\begin{array}{rclr}
\mathcal L^{(A)}(u_1) - (1-r)\Delta u_1 +\nabla p_1 &=& \div (\sigma_1) + a_1, & \text{on~$\omega_0$}\\[0.2cm]
\div (u_1) &=& 0, & \text{on~$\omega_0$}\\[0.2cm]
\mathcal L^{(B)}(u_1)+\mathcal L^{(C)}(\sigma_1) + \sigma_1 - \Di \Delta \sigma_1 &=& 2r \D(u_1), & \text{on~$\omega_0$}\\[0.2cm]
u_1 &=& - \displaystyle \lim_{Y\to \infty} \int_\T U_1, & \text{on~$\gamma^+$}\\[0.2cm] 
u_1 &=& 0, & \text{on~$\gamma^-_0$}\\[0.2cm] 
\partial_n \sigma_1 &=& 0, & \text{on~$\gamma^+$}\\[0.2cm]
\partial_n \sigma_1 &=& b_1, & \text{on~$\gamma^-_0$}.\\[0.2cm] 
\end{array}
\right.
$$

\noindent $\blacksquare$ {\bf Correction with boundary layer n$^{\mathrm o}$ 2:}
$$
\mathrm{PB}^{(2)}\left\{
\begin{array}{rclr}
- (1-r)\Delta U_2 +\nabla P_2 &=& F_2 - a_0, & \text{on~$\Omega$}\\[0.2cm]
 \div (U_2) &=& 0, & \text{on~$\Omega$}\\[0.2cm]
  U_2 &=& H \partial_y u_1|_{\gamma^-_0} - \frac{1}{2}H^2\partial_y^2 u_0|_{\gamma^-_0} & \text{on~$\Gamma$}\\[0.4cm]
-\Di\, \Delta \Sigma_2 &=& G_2, & \text{on~$\Omega$}\\[0.2cm]
 \partial_N \Sigma_2 &=& - H \partial_y^2\sigma_0|_{\gamma^-_0} + (\nabla H \cdot \nabla_x) \sigma_1|_{\gamma^-_0} & \\
&&\hspace{0.5cm}-H (\nabla H \cdot \nabla_x) \partial_y \sigma_0|_{\gamma^-_0} \\
&&\hspace{0.5cm}+ (\nabla H \cdot \nabla_x) \Sigma_1 - b_1 & \text{on~$\Gamma$}.
\end{array}
\right.
$$
with
$$\begin{array}{rcl}
F_2 &:=& \div (\Sigma_1) - \nabla_x P_1 + 2(1-r)\nabla_x \cdot \nabla U_1 - \Re ( u_0 \cdot \nabla U_1 ), \\
G_2 &:=& 2r \D(U_1) + 2\Di\, \nabla_x \cdot \nabla \Sigma_1 - \We \left( u_0 \cdot \nabla \Sigma_1 + g_a(\nabla U_1 , \sigma_0) \right).
\end{array}
$$

\begin{remark} With a suitable choice of~$a_{0}$ (resp.~$b_{1}$), the problem related to the velocity field (resp. constraint field) is well-posed and satisfies  the property with exponential decay towards a constant (resp.~$0$). Moreover,~$b_1$ does not depend on~$\sigma_1$!
\end{remark}

\noindent $\blacksquare$ {\bf Main correction at order $k\geq 2$:} 
$$
\mathrm{pb}^{(k)}\left\{
\begin{array}{rclr}
\mathcal L^{(A)}(u_k) - (1-r)\Delta u_k + \nabla p_k &=& \div (\sigma_k) + f_k + a_k & \text{on~$\omega_0$,} \\[0.2cm]
\div (u_k) &=& 0 & \text{on~$\omega_0$,} \\[0.2cm]
\mathcal L^{(B)}(u_k) + \mathcal L^{(C)}(\sigma_k) + \sigma_k - \Di \Delta \sigma_k &=& 2r \D(u_k) + g_k & \text{on~$\omega_0$,} \\[0.2cm]
u_k &=& -\lim_{Y\to \infty} \int_\T U_k & \text{on~$\gamma^+$,}\\[0.2cm]
u_k &=& 0 & \text{on~$\gamma^-_0$,}\\[0.2cm] 
\partial_n \sigma_k &=& 0 & \text{on~$\gamma^+$,}\\[0.2cm]
\partial_n \sigma_k &=& b_k & \text{on~$\gamma^-_0$}.
\end{array}
\right.
$$
where~$f_k$ and~$g_k$ only depend on solutions that were defined previously:
$$\begin{array}{ccl}
f_k &:=& -\Re\displaystyle \sum_{i=1}^{k-1} u_i \cdot \nabla u_{k-i},\\
g_k &:=& \We\displaystyle \sum_{i=1}^{k-1} (u_i \cdot \nabla \sigma_{k-i} + g(\nabla u_i,\sigma_{k-i})).
\end{array}$$

Higher order correction terms of the asymptotic expansion are defined as the solutions of the following elementary problems.\medskip

\noindent $\blacksquare$ {\bf Correction with boundary layer of order $k\geq 3$:}
$$
\mathrm{PB}^{(k)}\left\{
\begin{array}{rclr}
-(1-r)\Delta U_k + \nabla P_k &=& F_k - a_{k-2} & \text{on~$\Omega$,}\\[0.2cm]
\div (U_k) &=& 0 & \text{on~$\Omega$,}\\[0.2cm]
U_k &=& D_k & \text{on~$\Gamma$,}\\[0.4cm]
-\Di\, \Delta \Sigma_k &=& G_k & \text{on~$\Omega$,}\\[0.2cm]
\partial_N \Sigma_k &=& N_k - b_{k-1} & \text{on~$\Gamma$,}
\end{array}
\right.
$$
where~$F_k$, $D_k$, $G_k$ and~$N_k$ only depend on solutions that were defined previously.
$$
\begin{array}{ccl}
F_k &:=& \div (\Sigma_{k-1}) + 2(1-r) \nabla_x \cdot \nabla U_{k-1} \\[0.2cm]
&& \quad + (1-r) \Delta_x U_{k-2} - \nabla_x P_{k-1} + \div_x (\Sigma_{k-2}) \\
&& \quad - \Re \displaystyle \sum_{i=0}^{k-2} \left( (u_i + U_i)\cdot (\nabla U_{k-1-i} + \nabla_x U_{k-2-i}) + U_i\cdot \nabla u_{k-i-2} \right), \\[0.2cm]
D_k &:=& -\displaystyle \sum_{p=1}^{k} \frac{(-1)^p}{p!} H^p \partial_y^{(p)} u_{k-p}|_b, \\
G_k &:=&\displaystyle 2r \D(U_{k-1})+ 2r \D_x(U_{k-2})-\Sigma_{k-2}+\displaystyle 2\Di\, \nabla_x \cdot \nabla \Sigma_{k-1} +\Di\, \Delta_x \Sigma_{k-2} \\
&&\begin{array}{l}
\begin{array}{ll}-\We\Bigg( & \displaystyle\sum_{i=0}^{k-2} u_i \cdot (\nabla \Sigma_{k-1-i} + \nabla_x \Sigma_{k-i-2})\\
&+ U_i \cdot (\nabla \sigma_{k-i-2} + \nabla \Sigma_{k-1-i} + \nabla_x \Sigma_{k-i-2})\\
&\displaystyle+ g_a(\nabla u_i , \Sigma_{k-i-2}) + g_a(\nabla U_{i+1} + \nabla_x U_i , \sigma_{k-i-2} + \Sigma_{k-i-2})\Bigg),\end{array}
\end{array}\\[0.4cm] 
N_k &:=& \displaystyle (\nabla H \cdot \nabla_x ) \sigma_{k-1}|_{\gamma^-_0} + (\nabla H \cdot \nabla_x) \Sigma_{k-1} \\
&& \hspace{1cm} \displaystyle + \sum_{p=1}^{k-1} \frac{(-1)^p}{p!} H^p \partial_y^{(p)} \left( (\nabla H \cdot \nabla_x) \sigma_{k-1-p} + \partial_y \sigma_{k-1-p} \right)|_{\gamma^-_0}.
 \end{array}$$

\section{Analysis of the elementary problems: well-posedness and properties of the solutions}\label{Section4}

The elementary problems related to the boundary layer correctors with respect to the velocity take the following form:
$$
\mathrm{PB}^{(\mathrm{ref}1)}\left\{
\begin{array}{rclr}
-(1-r)\Delta U + \nabla P &=& F & \text{on~$\Omega$},\\[0.2cm]
\div (U) &=& 0 & \text{on~$\Omega$},\\[0.2cm]
U &=& D & \text{on~$\Gamma$},
\end{array}
\right.
$$
with $F\in L^2(\Omega)$ and $D\in L^2(\T^d)$.
 
\begin{notation}
We denote~$\widehat{F_j}$ the following Fourier coefficients of a function~$F$ defined on $\T^d \times (0,+\infty)$:
\begin{equation*}
\begin{array}{c}
F(X,Y) = \displaystyle \sum_{j\in \Z^d} \widehat{F_j}(Y)e^{ 2\pi \mathrm i j \cdot X}.
\end{array}
\end{equation*}
\end{notation}

\begin{notation}
All constants depending only on the domain and on physical constants will be considered harmless, there will be denoted~$C$.
In the same way, we will use the notation~$Q(Y)$ to denote any polynomial with coefficients depending only on the domain or on physical constant.
In particular the quantities~$C$ and~$Q(Y)$ do not depend on the variables~$x$ or~$X$.
\end{notation}

\begin{proposition}\label{prop:blp1}
If the following conditions are satisfied
$$
\begin{aligned}
& \left\| \widehat{F_0}(Y) \right\| \leq Q(Y)\, \e^{-Y}, \quad && \text{for all $Y>0$},\\
& \left\| \widehat{F_j}(Y) \right\| \leq Q(Y)\, \e^{-\|j\| Y}, \quad && \text{for all $j\in \mathbb Z^d\setminus\{0\}$ and $Y>0$},
\end{aligned}
$$
then~$\mathrm{PB}^{(\mathrm{ref}1)}$  admits a unique solution $(U,P)$ satisfying $\nabla U\in L^2(\Omega)$ and $P\in L^2(\Omega)$.
\begin{enumerate}
\item There exists $U_\infty \in \R^d$ such that
$$
\begin{aligned}
& \left\| \widehat{U_0}(Y)-U_\infty \right\|  + \left| \widehat{P_0}(Y) \right| \leq Q(Y)\, \e^{-Y}, \quad && \text{for all $Y>0$},\\
& \left\| \widehat{U_0}'(Y) \right\|  + \left| \widehat{P_0}'(Y) \right| \leq Q(Y)\, \e^{-Y}, \quad && \text{for all $Y>0$}.
\end{aligned}
$$
\item We have
$$
\|U(X,Y) - U_\infty \| + |P(X,Y)| \leq Q(Y)\, \e^{-Y}, \quad \text{for all $X \in \T^d$, $Y>0$},
$$
and, for all $\ell\geq 1$,
$$
\|\partial_{X}^{\ell} U(X,Y) \| + |\partial_{X}^{\ell} P(X,Y)| \leq Q(Y)\, \e^{-Y}, \quad \text{for all $X \in \T^d$, $Y>0$}.
$$
\end{enumerate}
\end{proposition}

\begin{remark}
In this Proposition, the vector~$U_\infty$ can be identified as the limit of~$\widehat{U_0}(Y)$ when~$Y$ goes to $+\infty$:
$$U_\infty := \lim_{Y\to\infty} \int_\T U \in \R^d.$$
\end{remark}

\begin{proof}
Let us introduce the following decomposition of the vectors $U=(U^{(1)},U^{(2)})\in \R^d\times \R$ and $F=(F^{(1)},F^{(2)})\in \R^d\times \R$.
Then, we pass to the Fourier transform with respect to~$X$. Passing in the Fourier regime, equations satisfied by $(U,P)$ inside the domain~$\Omega$, in problem~$\mathrm{PB}^{(1)}$, can be translated into
\begin{equation}\label{syst-Fourier0}
\left\{
\begin{array}{rclll}
{\|j\|^2} \widehat{U^{(1)}_j} -  \widehat{U^{(1)}_j}'' +  \mathrm i j \widehat{P_j} &\!=\!& \widehat{F^{(1)}_{j}} \qquad & \text{on $\{Y>0\}$} \qquad & \forall j\in \Z^d,\\
{\|j\|^2} \widehat{U^{(2)}_j} -  \widehat{U^{(2)}_j}'' + \widehat{P_j}' &\!=\!& \widehat{F^{(2)}_{j}} & \text{on $\{Y>0\}$} & \forall j\in \Z^d,\\
\mathrm i j \cdot \widehat{U^{(1)}_j} + \widehat{U^{(2)}_j}' &\!=\!& 0 & \text{on $\{Y>0\}$} & \forall j\in \Z^d,\\
\end{array}
\right.
\end{equation}
where~$\widehat{U^{(1)}_j}'$, $\widehat{U^{(2)}_j}'$ and~$\widehat{P_j}$ belong to $L^2(0,+\infty)$.
Now we solve the Fourier problem and describe the behavior of the solution of the Stokes problem.

\begin{enumerate}
\item Let us discuss the case $j=0$. The system reduces to
\begin{equation*}
\left\{
\begin{array}{rllll}
-  \widehat{U^{(1)}_0}'' =  \widehat{F^{(1)}_{0}}
& ~ \text{with $\widehat{U^{(1)}_0}' \in L^2(0,+\infty)$,} \\
-  \widehat{U^{(2)}_0}'' + \widehat{P_0}' = \widehat{F^{(2)}_{0}}
& ~ \text{with $\widehat{P_0} \in  L^2(0,+\infty)$,}\\
\widehat{U^{(2)}_0}' = 0
& ~ \text{with $\widehat{U^{(2)}_0}' \in  L^2(0,+\infty)$.}
\end{array}
\right.
\end{equation*}
By integration, this leads us to the following equalities 
\begin{equation*}
\begin{aligned}
& \widehat{U^{(1)}_{0}}'(Y) = \int_{Y}^{+\infty}\widehat{F^{(1)}_{0}}(\xi)\, \mathrm d\xi,\\
& \widehat{U^{(2)}_{0}}'(Y) = 0,\\
& \widehat{U^{(1)}_{0}}(Y) = {\widehat{U^{(1)}_{0}}(0)-\displaystyle\int_{0}^{+\infty}\left(\displaystyle\int_{Z}^{+\infty}\widehat{F^{(1)}_{0}}(\xi)\, \mathrm d\xi\right)\, \mathrm dZ} \\
& \hspace{4cm} + {\displaystyle\int_{Y}^{+\infty}\left(\displaystyle\int_{Z}^{+\infty}\widehat{F^{(1)}_{0}}(\xi)\, \mathrm d\xi\right)\, \mathrm dZ},\\
& \widehat{U^{(2)}_{0}}(Y) = \widehat{U^{(2)}_{0}}(0),\\
& \widehat{P_{0}}(Y) = -\displaystyle\int_{Y}^{+\infty}\widehat{F^{(2)}_{0}}(Z)\, \mathrm dZ.
\end{aligned}
\end{equation*}
By assumption on the source term, we have
$$
\begin{aligned}
& \left| \displaystyle\int_{Y}^{+\infty}\displaystyle\int_{Z}^{+\infty}\widehat{F^{(1)}_{0}}(\zeta)\, \mathrm d\zeta\, \mathrm dZ \right| \leq Q(Y)\, \e^{-Y} \\
& \left| \displaystyle\int_{Y}^{+\infty}\widehat{F^{(\ell)}_{0}}(Z)\, \mathrm dZ \right| \leq Q(Y)\, \e^{-Y}, \quad \ell = 1,2.
\end{aligned}
$$
Defining $U_\infty:=(U_\infty^{(1)},U_\infty^{(2)}) \in \R^d\times \R$ as
$$\begin{array}{rcl}
U_\infty^{(1)} &:=& \widehat{U^{(1)}_{0}}(0)-\displaystyle\int_{0}^{+\infty}\left(\displaystyle\int_{Z}^{+\infty}\widehat{F^{(1)}_{0}}(\xi)\, \mathrm d\xi\right)\, \mathrm dZ,\\ 
U_\infty^{(2)} &:=& \widehat{U^{(2)}_{0}}(0),
\end{array}
$$
we obtain, for all $Y>0$,
$$
\begin{aligned}
|\widehat{U^{(1)}_0}(Y)- U_\infty^{(1)} | + |\widehat{U^{(2)}_0}(Y)- U_\infty^{(2)} | + |\widehat{P_0}(Y) | \leq Q(Y)\, \e^{-Y} , \\
|\widehat{U^{(1)}_0}'(Y)| + |\widehat{U^{(2)}_0}'(Y)| + |\widehat{P_0}'(Y) | \leq Q(Y)\, \e^{-Y}.
\end{aligned}
$$
\item Let us discuss the case $j\neq 0$. If the source terms~$\widehat{F^{(1)}_j}$ and~$\widehat{F^{(2)}_j}$ were written as $Q(Y) \e^{-\|\mathbf j\|Y}$ then the estimate is proved in~\cite[Appendix B]{CM12}. In the present case, the source terms only satisfy 
$$
\| \widehat{F^{(1)}_j}(Y) \| \leq Q(Y) \e^{-\|\mathbf j\|Y},\quad \forall Y>0,
$$
$$
| \widehat{F^{(2)}_j}(Y) | \leq Q(Y) \e^{-\|\mathbf j\|Y},\quad \forall Y>0,
$$
and the proof can be adapted from~\cite{CM12} by using a comparison principle.
In particular, we can show that
$$
\begin{aligned}
& \left\| \widehat{U_j}(Y) \right\| + \left| \widehat{P_j}(Y) \right| \leq Q(Y)\, \e^{-\|j\| Y}, \quad && \text{for all $j\in \mathbb Z^d\setminus\{0\}$ and $Y>0$}.
\end{aligned}
$$

\end{enumerate}
\end{proof}

The elementary problems related to the boundary layer correctors with respect to the elastic constraint take the following form:
$$
\mathrm{PB}^{(\mathrm{ref}2)}\left\{
\begin{array}{rclr}
- \Delta \Sigma &=& G & \text{on~$\Omega$},\\[0.2cm]
\partial_N \Sigma &=& N & \text{on~$\Gamma$},
\end{array}
\right.
$$
with $G \in L^2(\Omega)$ and $N\in L^2(\T^d)$.

We note that the problems related to the boundary layer correctors are matricial problems: the unknowns~$\Sigma_k$ have matricial values. The elementary problem that we analyze here is a scalar case but can easily be extend component by component.

\begin{proposition}\label{prop:wp2}
Assume that
$$
\displaystyle \int_{\{Y<0\}} G(X,Y)\, \mathrm{d}X\, \mathrm dY=\displaystyle \int_{\T^d} N(X)\, \mathrm{d}X.
$$
If the following conditions are satisfied
$$
\begin{aligned}
& \left| \widehat{G_0}(Y) \right| \leq Q(Y)\, \e^{-Y}, \quad && \text{for all $Y>0$},\\
& \left| \widehat{G_j}(Y) \right| \leq Q(Y)\, \e^{-\|j\| Y}, \quad && \text{for all $j\in \mathbb Z^d\setminus\{0\}$ and $Y>0$},
\end{aligned}
$$
then~$\mathrm{PB}^{(\mathrm{ref}2)}$  admits a unique solution $\Sigma\in L^2(\Omega)$.
Moreover, we have
$$
\begin{aligned}
& \left| \widehat{\Sigma_0}(Y) \right| \leq Q(Y)\, \e^{-Y}, \quad && \text{for all $Y>0$},\\
& \left| \widehat{\Sigma_j}(Y) \right| \leq Q(Y)\, \e^{-\|j\| Y}, \quad && \text{for all $j\in \mathbb Z^d\setminus\{0\}$ and $Y>0$}.
\end{aligned}
$$
In particular, we have
$$
| \Sigma(X,Y) | \leq Q(Y)\, \e^{-Y}, \quad \text{for all $X \in \T^d$, $Y>0$}.
$$
\end{proposition}

\begin{proof}  If the source terms satisfy $\widehat{G_{0}}=0$ and, for $j\geq 1$, $\widehat{G_j}(Y)=Q(Y) \e^{-\|\mathbf j\|Y}$, then the estimate is proved in~\cite[Lemma 2.2]{Chu12}. In the present case, the source terms only satisfy 
$$
| \widehat{G_{0}}(Y) | \leq Q(Y)\, \e^{-Y}
\quad \text{and} \quad
| \widehat{G_j}(Y) | \leq Q(Y) \e^{-\|j\|Y}.
$$
The estimate on~$\widehat{\Sigma_{0}}$ can be obtained by straightforward integration. For $j\neq 0$, the proof can be adapted from~\cite{Chu12} by using a comparison principle.
\end{proof}

\paragraph{Applications: analysis of problems $\mathrm{PB}^{(k)}$.}

Let us recall that the definition of the boundary layer correction problems~$\mathrm{PB}^{(k)}$ need to specify the value of~$a_{k-2}$ and~$b_{k-1}$. Let us first describe how to determine~$a_{k-2}$. In order to apply Proposition \ref{prop:blp1}, we need to impose that the source term in the momentum equation of~$\mathrm{PB}^{(k)}$, namely $F:=F_k- a_{k-2}$, satisfies a sharp decrease for each Fourier mode: 
\begin{itemize}
\item Averaging this source term~$F$ with respect to~$X$ gives
$$\begin{array}{rcl}
\widehat{F}_0(Y)&=&{\displaystyle\int_{\T^d} (F_k(X,Y)- a_{k-2})\, \mathrm{d}X}\\
&=&\left( \displaystyle\int_{\T^d} F_k(X,Y) \, \mathrm dX - \displaystyle\lim_{Y\to +\infty}\int_{\T^d} F_k(X,Y)\, \mathrm dX \right)\\
&&\hspace{2cm} + \left( \displaystyle\lim_{Y\to +\infty}\int_{\T^d} F_k(X,Y)\, \mathrm dX-a_{k-2}\right).
\end{array}$$
On one hand, $F_k$ is composed of elementary solutions $(U_i,\widetilde{\Sigma}_i,P_i)_{1\leq i\leq k-1}$ which, by induction, satisfy the expected decreasing behavior. We have
$$
\left|\displaystyle\int_{\T^d} F_k(X,Y) - \lim_{Y\to +\infty}\int_{\T^d} F_k(X,Y)\, \mathrm{d}X\right| \leq Q(Y)\, \e^{-Y}.
$$
On the other hand, in order to satisfy the assumption needed to apply Proposition \ref{prop:blp1}, we impose
\begin{equation}\label{eq:ak}
a_{k-2} = \lim_{Y\to +\infty}\int_{\T^d} F_k(X,Y)\, \mathrm dX.
\end{equation}
\item By induction on~$k$, the Fourier coefficients~$\widehat{F_{j}}$, for $j\neq 0$, of the source term $F= F_k - a_{k-2}$ satisfy
$$
\left\| \widehat{F_j}(Y) \right\| \leq Q(Y)\, \e^{-\|j\| Y},\quad \forall j\in \mathbb Z^d\setminus\{0\},\quad \forall Y>0.
$$
\end{itemize}

Let us now describe how to determine~$b_{k-1}$.
In order to apply Proposition~\ref{prop:wp2}, we need to impose the compatibility condition between the source term in the Laplace equation of~$\mathrm{PB}^{(k)}$, namely~$G:=\Di^{-1}G_k$, and the Neumann boundary term $N = N_k- b_{k-1}$:
\begin{equation}\label{eq:bk}
b_{k-1}=\int_{\T^d} N_k(X)\, \mathrm dX - \Di^{-1} \int_{Y<0} G_k(X,Y)\, \mathrm dX\, \mathrm dY.
\end{equation}
Besides, by induction on~$k$, the Fourier coefficients~$\widehat{G_j}$ of the source term~$G:=\Di^{-1}G_k$ satisfy
$$\begin{array}{rcll}
\left\| \widehat{G_0}(Y) \right\| &\leq Q(Y)\, \e^{-Y},& \forall Y>0,\\
\left\| \widehat{G_j}(Y) \right\| &\leq Q(Y)\, \e^{-\|j\| Y},& \forall j\in \mathbb Z^d\setminus\{0\},& \forall Y>0.\phantom{\Bigg(}
\end{array}
$$

\section{Error estimates}\label{Section5}

\subsection{Remainder}

The asymptotic expansion truncated at a given order leads us to introduce the so-called remainder $(\cR,\cQ,\cS)$:
\begin{equation*}
\begin{array}{lllll}
u(x) &=& \displaystyle
\sum_{j=0}^{N} \varepsilon^j \left[ u_j\left(x\right) + U_j\left(x,\displaystyle \frac{x}{\varepsilon}\right) \right] &+& \varepsilon^N \mathcal R(x), \\[0.25cm]
p(x) &=& \displaystyle
\sum_{j=0}^{N} \varepsilon^{j} \left[ p_j\left(x\right) + P_{j+1}\left(x,\displaystyle \frac{x}{\varepsilon}\right) \right] &+& \varepsilon^N \cQ(x), \\[0.25cm]
\sigma(x) &=& \displaystyle
\sum_{j=0}^{N} \varepsilon^j \left[ \sigma_j\left(x\right) + \Sigma_j\left(x,\displaystyle \frac{x}{\varepsilon}\right) \right] &+& \varepsilon^N \cS(x). \\[0.25cm]
\end{array}
\end{equation*}
We aim at establishing estimates on the remainder, {\em at any order}.

Applying the Oldroyd operator to the remainder $(\mathcal R,\cQ,\cS)$ and considering the properties of the elementary solutions, we get the following set of equations
\begin{itemize}
\item[$\bullet$] momentum equation, in~$\omega_\varepsilon$: 
$$\Re\, \left(\varepsilon^N \cR\cdot\nabla\cR+\mathcal L_{\varepsilon}^{(A)}(\cR)\right)- (1-r)\Delta \cR + \nabla \cQ =\div(\cS)+ \mathcal F_\varepsilon,$$
\item[$\bullet$] continuity equation, in~$\omega_\varepsilon$:
$$\mathrm{div}(\cR) = 0,$$
\item[$\bullet$] constitutive equation, in~$\omega_\varepsilon$:
$$\begin{array}{r}
\We\, \left(\varepsilon^N \left(\cR \cdot \nabla\cS + g_{a}(\nabla \cR, \cS)\right)+ \mathcal L_{\varepsilon}^{(B)}( \cR )+ \mathcal L_{\varepsilon}^{(C)}( \cS )\right) + \cS - \Di \Delta \cS\\ 
= 2r \D(\cR) + \mathcal G_\varepsilon,
\end{array}$$
\item[$\bullet$] boundary conditions on the velocity, on $\gamma_\varepsilon^- \cup \gamma^+$:
$$\cR = \mathcal D_\varepsilon^{(\pm)},$$
\item[$\bullet$] boundary conditions on the elastic constraint, on $\gamma_\varepsilon^- \cup \gamma^+$: 
$$\partial_{\mathbf n} \cS = \mathcal N_\varepsilon^{(\pm)},$$
\end{itemize}
where operators are defined as follows.
$$
\begin{array}{ccl}
\mathcal L_{\varepsilon}^{(A)}(\cR)&=&\cR\cdot\left( \displaystyle \sum_{k=0}^{N-1}\varepsilon^k(\nabla u_k + \nabla_x U_k + \nabla U_{k+1}) \right)
\\
&&+ \varepsilon^N\, \cR \cdot (\nabla u_N + \nabla_x U_N) + \left(\displaystyle\sum_{k=0}^N \varepsilon^k( u_k + U_k) \right) \cdot \nabla\cR,\\
\mathcal L_{\varepsilon}^{(B)}( \cR )&=&\cR \cdot\left(\displaystyle\sum_{k=0}^{N-1} \varepsilon^k (\nabla \sigma_k + \nabla_x \Sigma_k + \nabla \Sigma_{k+1}) \right)\\
&&+ \varepsilon^N\, \cR \cdot (\nabla \sigma_N + \nabla_x \Sigma_N) + g_a \left(\nabla \cR,\displaystyle\sum_{k=0}^N \varepsilon^k (\sigma_k + \Sigma_k) \right),\\
\mathcal L_{\varepsilon}^{(C)}( \cS )&=&\left(\displaystyle\sum_{k=0}^N \varepsilon^k (u_k + U_k)\right) \cdot \nabla \cS+ \varepsilon^N g_a (\nabla u_N + \nabla_x U_N , \cS)\\
&&+ g_a \left( \displaystyle\sum_{k=0}^{N-1}\varepsilon^k( \nabla u_k + \nabla_x U_k + \nabla U_{k+1} ),\cS\right).\\
\end{array}
$$
The source terms are defined by 
$$
\begin{array}{rcl}
\mathcal F_\varepsilon&=&-\Re \left(\displaystyle\sum_{k=1}^N ( u_k + U_k ) \cdot (\nabla u_{N-k} + \nabla_x U_{N-k} + \nabla U_{N+1-k}) \right)\\[0.2cm]
&&\quad- \Re\, u_0 \cdot (\nabla u_N + \nabla_x U_N)+ (1-r)\Delta_x U_N + (1-r) \Delta_x U_N\\[0.2cm]
&& \quad - \nabla p_N - \nabla_x P_{N+1} + \div (\sigma_N) + \div_x (\Sigma_N),\\[0.3cm]
\mathcal G_\varepsilon&=&2r \D(u_N) + 2r \D_x(U_N) - \sigma_N - \Sigma_N + \Di\, \Delta \sigma_N + \Di\, \Delta_x \Sigma_N\\[0.2cm]
&&\quad - \We \displaystyle\sum_{k=1}^N \Bigg( (u_k + U_k) \cdot ( \nabla \sigma_{N-k} + \nabla_x \Sigma_{N-k} + \nabla \Sigma_{N+1-k} ) \Bigg)\\[0.cm]
&&\quad-\We\, u_0 \cdot ( \nabla \sigma_N + \nabla_x \Sigma_N )\\[0.2cm]
&&\quad - \We \displaystyle\sum_{k=0}^{N-1} \Bigg(g_a (\nabla u_k + \nabla_x U_k + \nabla U_{k+1} , \sigma_{N-k} + \Sigma_{N-k} ) \Bigg)\\[0.2cm]
&&\quad -\We\, g_a ( \nabla u_N + \nabla_x U_N , \sigma_0 ),\\
\end{array}
$$
and the contributions to the boundary relation are given by
$$
\begin{array}{rcl}
\mathcal D_\varepsilon^{(+)}&=&\displaystyle\sum_{k=0}^N \varepsilon^{k-N} \left(\lim_{Y\to\infty} \int_\T U_k - U_k|_{Y=1/\varepsilon}\right),\\
\mathcal D_\varepsilon^{(-)}&=&-\varepsilon^{-N} \displaystyle\sum_{k= 0}^{N} \varepsilon^k \displaystyle\sum_{i=N-k+1}^{+\infty}\displaystyle\frac{(-\varepsilon H)^i}{i!}\partial_{y}^{(i)} u_k|_{\gamma^-_0},\\
\mathcal N_\varepsilon^{(+)}&=&-\displaystyle\sum_{k=0}^{N-1} \varepsilon^{k-N} \partial_{Y}\Sigma_{k+1}|_{y=1,Y=1/\varepsilon},\\
\mathcal N_\varepsilon^{(-)}&=&-\varepsilon^{-N} \displaystyle\sum_{k= 0}^{N} \varepsilon^{k}\displaystyle\sum_{i=N-k+1}^{+\infty}\displaystyle\frac{(-\varepsilon H)^i}{i!}\partial_{y}^{(i)}((\nabla H \cdot \nabla_x) \sigma_k + \partial_{y}\sigma_{k})|_{\gamma^-_0}.
\end{array}
$$

\begin{proposition}\label{prop:estimateST}
The following estimates hold:
$$
\|\mathcal F_\varepsilon\|_{L^2} \leq C,\qquad  \|\mathcal G_\varepsilon\|_{L^2} \leq C.
$$
For all $\ell \geq 0$, for all $x\in \T^d$, we have
$$
\begin{aligned}
& |\nabla_x^\ell \mathcal D_\varepsilon^{(+)}(x)| \leq Q\left(\frac{1}{\varepsilon}\right) \, \e^{-\frac{1}{\varepsilon}},
\qquad
&& |\nabla_x^\ell \mathcal D_\varepsilon^{(-)}(x)| \leq C\, \varepsilon^{1-\ell},\\
& |\nabla_x^\ell \mathcal N_\varepsilon^{(+)}(x)| \leq Q\left(\frac{1}{\varepsilon}\right) \, \e^{-\frac{1}{\varepsilon}},
\qquad 
&& |\nabla_x^\ell \mathcal N_\varepsilon^{(-)}(x)| \leq C\, \varepsilon^{1-\ell}.
\end{aligned}
$$
\end{proposition}

\begin{proof}
\begin{itemize}
\item[]
\item The estimates for~$\mathcal F_\varepsilon$ and~$\mathcal G_\varepsilon$ are obvious.
\item By using Proposition~\ref{prop:blp1}, we know that for each integer~$k$ we have
$$
\left\|U_k \left( x,\frac{x}{\varepsilon},\frac{1}{\varepsilon} \right) - \lim_{Y\to\infty} \int_\T U_k \right\| \leq Q\left(\frac{1}{\varepsilon}\right) \, \e^{-\frac{1}{\varepsilon}}, \quad \text{for all $x \in \T^d$},
$$
and, for all $\ell\geq 1$,
$$
\left\|\nabla_X^{\ell} U_k \left( x,\frac{x}{\varepsilon},\frac{1}{\varepsilon}\right) \right\| \leq Q\left(\frac{1}{\varepsilon}\right) \, \e^{-\frac{1}{\varepsilon}}, \quad \text{for all $x \in \T^d$}.
$$
We immediately deduce that for all $\ell\geq 0$ and for all $x\in \T^d$ we have
$$
|\nabla_x^\ell \mathcal D_\varepsilon^{(+)}(x)| \leq Q\left(\frac{1}{\varepsilon}\right) \, \e^{-\frac{1}{\varepsilon}}.
$$
\item We estimate $\nabla_x^\ell \mathcal D_\varepsilon^{(-)}$ remarking that, using the Taylor formulae, for each integer~$k$ we can write
$$
\sum_{i=N-k+1}^{+\infty} \frac{(-\varepsilon H)^i}{i!}\partial_y^{(i)}u_k(x,0)
= \frac{(-\varepsilon H)^{N-k+1}}{(N-k+1)!}\partial_{y}^{(N-k+1)} u_k(x,\xi_k),
$$
with $\xi_k\in [0,-\varepsilon H(x/\varepsilon)]$.
That implies
$$
\mathcal D_\varepsilon^{(-)}(x,-\varepsilon H(x/\varepsilon))=-\varepsilon \displaystyle \sum_{k=0}^N \displaystyle\frac{(-H)^{N-k+1}}{(N-k+1)!}\partial_{y}^{(N-k+1)} u_k(x,\xi_k).
$$
Now, since~$\mathcal D_\varepsilon^{(-)}$ is a finite sum, the estimate directly follows from its analysis.
\item Finally, the estimate on $\nabla_x^\ell \mathcal N_\varepsilon^{(\pm)}$ is based on the same arguments, noticing that the Neumann data, for the remainder, can be reduced to a {\em finite} sum of boundary terms, namely
$$\begin{array}{l}
\mathcal N_\varepsilon^{(N)}(x,-\varepsilon H(x/\varepsilon)) \\
\hspace{0.5cm}= \displaystyle\varepsilon \sum_{k=0}^N \frac{(-H)^{N-k+1}}{(N-k+1)!} \partial_y^{(N-k+1)}((\nabla H \cdot \nabla_x) \sigma_k + \partial_y \sigma_k)(x,\zeta_k).
\end{array}$$
\end{itemize}
\end{proof}

\subsection{Lift procedure}

Well-posedness of the set of equations satisfied by the remainder is obtained by means of construction. Let us point out the fact that the remainder satisfies a diffusive Oldroyd-type system. 

In this step, we aim at modifying the set of equations by using a lift procedure in order to deal with {\em homogeneous} boundary conditions and preserve the {\em homogeneous} incompressibility condition. 

\begin{definition}
Let $\tau \in C^\infty(\mathbb R)$ be such that 
$$\tau(y)=\left\{
\begin{array}{ll}
0 \quad & \textnormal{if $y\leq 0$,}\\
1       & \textnormal{if $y\geq 1$.}
\end{array}\right.
$$
We define~$\mathcal R_{\mathrm{bound}}$ and~${\mathcal S}_{\mathrm{bound}}$ as
\begin{equation}\label{bound-definition}
\begin{aligned}
& {\cR}_{\mathrm{bound}}(x,y) = \tau(y) \, \mathcal D_\varepsilon^{(+)}(x) + (1-\tau(y)) \, \mathcal D_\varepsilon^{(-)}(x), \\
& {\cS}_{\mathrm{bound}}(x,y) = \tau(y) \, \mathcal N_\varepsilon^{(+)}(x) + (1-\tau(y)) \, \mathcal N_\varepsilon^{(-)}(x), \\
\end{aligned}
\end{equation}
and we define~${\cR}_{\mathrm{div}}$ as a solution of
$$
\left\{
\begin{array}{rcll}
\mathrm{div} ({{\cR}_{\mathrm{div}}}) & = & -\mathrm{div} ({{\cR}_{\mathrm{bound}}}) \quad & \text{in~$\omega_\varepsilon$},\\
{{\cR}_{\mathrm{div}}} & =&0 & \text{on $\gamma_\varepsilon^- \cup \gamma^+$}.
\end{array}
\right.
$$
The lift velocity and constraint fields are thus defined as:
$$
{\cR}_{\mathrm{lift}}:={\cR}_{\mathrm{bound}}+{\cR}_{\mathrm{div}},\qquad {\cS}_{\mathrm{lift}}:={\cS}_{\mathrm{bound}}
$$
and the {\em lifted} velocity and constraint field are thus defined as:
$$
\widetilde{\cR}:={\cR}-{\cR}_{\mathrm{lift}},\quad \widetilde{\cQ}:=\cQ, \quad  \widetilde{\cS}:={\cS}-{\cS}_{\mathrm{lift}}.
$$
\end{definition}

\begin{remark} The lifted remainder $(\widetilde{\cR},\widetilde{\cS})$ satisfies a system which is identical to the system satisfied by $({\cR},{\cS})$, up to some modifications:
\begin{itemize}
\item the boundary conditions are {\em homogeneous}~;
\item the incompressibility condition is (still) homogeneous~;
\item the source terms~$\mathcal F_{\varepsilon}$, $\mathcal G_{\varepsilon}$, and linear operators~$\mathcal L^{(A)}$, $\mathcal L^{(B)}$, $\mathcal L^{(C)}$ have been (slightly) modified (see further, page~\pageref{coeff-modif})
\end{itemize}
\end{remark}

\begin{remark} The definition of~${\mathcal R}_{\mathrm{div}}$ is guaranteed by the following result, due to Bogovskii \cite{Bog79} (see also \cite{BF06}):

\begin{proposition}[Bogovskii]\label{prop:bogovskii}
If $\mathcal H \in {H}^m(\omega_\varepsilon)$, $m\geq 0$, is such that
$$
\displaystyle\int_{\Omega_{\varepsilon}} \mathcal H=0,
$$
then there exists a solution $\widetilde{\cR} \in H^{m+1}(\omega_\varepsilon)$ of
$$
\left\{
\begin{array}{rcll}
\mathrm{div} (\widetilde{\cR}) & = & \mathcal H \quad & \text{in~$\omega_\varepsilon$},\\
\widetilde{\cR} & =& 0 & \text{on $\gamma_\varepsilon^- \cup \gamma^+$},
\end{array}
\right.
$$
such that
\begin{equation*}
\| \nabla \widetilde{\cR} \|_{H^m(\omega_\varepsilon)}
\leq
C \, \| \mathcal H \|_{H^m(\omega_\varepsilon)}.
\end{equation*}
\end{proposition}
Thus, the existence of such a lift function~$\widetilde{\cR}$ relies on the identity $\int_{\Omega_{\varepsilon}} \mathcal H=0$ with $\mathcal H=-\mathrm{div} ({{\cR}_{\mathrm{bound}}})$. By the Stokes formula, we have
$$
\displaystyle\int_{\Omega_{\varepsilon}} \mathrm{div} ({{\cR}_{\mathrm{bound}}})=\displaystyle\int_{\gamma_{\varepsilon}^-} \mathcal D_\varepsilon^{(-)} \cdot \mathbf n+\displaystyle\int_{\gamma^+} \mathcal D_\varepsilon^{(+)} \cdot \mathbf n=\displaystyle\int_{\Omega_{\varepsilon}} \mathrm{div} ({{\cR}})=0.
$$ 
\end{remark}

\begin{proposition}\label{prop:estime-lift}
The following estimates hold:
$$
\| \mathcal R_{\mathrm{lift}}\|_{H^2}\leq C\, \varepsilon^{-1},
\qquad
\| \mathcal S_{\mathrm{lift}}\|_{H^2}\leq C\, \varepsilon^{-1}.
$$
\end{proposition}

\begin{proof}
\begin{itemize}
\item[]
\item By definition of~$\mathcal R_{\mathrm{bound}}$ and~$\mathcal S_{\mathrm{bound}}$ (see~\ref{bound-definition}), the estimates
$$
\| \mathcal R_{\mathrm{bound}}\|_{H^2}\leq C\, \varepsilon^{-1},
\qquad
\| \mathcal S_{\mathrm{bound}}\|_{H^2}\leq C\, \varepsilon^{-1}.
$$
directly follows from the Proposition~\ref{prop:estimateST}.
\item Following the Bogovskii inequality (see Proposition~\ref{prop:bogovskii}) and the Poincaré inequality, we have
$$
\begin{array}{rcll}
\| \mathcal R_{\mathrm{div}}\|_{H^2} & \leq C\, & \| \mathcal R_{\mathrm{bound}}\|_{H^2}.
\end{array}
$$ 
\item Since $\mathcal R_{\mathrm{lift}} = \mathcal R_{\mathrm{bound}}+\mathcal R_{\mathrm{div}}$ and $\mathcal S_{\mathrm{lift}} = \mathcal S_{\mathrm{bound}}$, the two previous steps imply the result of the Proposition~\ref{prop:estime-lift}.
\end{itemize}
\end{proof}

Now estimates for a homogeneous (w.r.t. boundary conditions and incompressibility condition) system have to be established. This is the purpose of the next subsection.

\subsection{Estimate on the remainder}

The {\em lifted} remainder $(\widetilde{\cR},\widetilde{Q},\widetilde{\cS})$ satisfies the following system:
\begin{itemize}
\item[$\bullet$] momentum equation, in~$\omega_\varepsilon$: 
$$\varepsilon^N\, \Re\, \widetilde{\cR} \cdot \nabla \widetilde{\cR} + \widetilde{\mathcal L_{\varepsilon}^{(A)}}(\widetilde{\cR})- (1-r)\Delta \widetilde{\cR} + \nabla \widetilde{Q} = \div(\widetilde{\cS})+ \widetilde{\mathcal F_\varepsilon},$$
\item[$\bullet$] constitutive equation, in~$\omega_\varepsilon$:
$$\begin{array}{r}
\varepsilon^N\, \We\, \left( \widetilde{\cR} \cdot \nabla \widetilde{\cS} + g_a(\nabla \widetilde{\cR}, \widetilde{\cS}) \right) + \widetilde{\mathcal L_{\varepsilon}^{(B)}}(\widetilde{\cR}) + \widetilde{\mathcal L_{\varepsilon}^{(C)}}(\widetilde{\cS}) + \widetilde{\cS} - \Di \Delta \widetilde{\cS} \\
= 2r \D(\widetilde{\cR})+ \widetilde{\mathcal G_\varepsilon}
\end{array}$$
\item[$\bullet$] the {\em homogeneous} incompressibility condition, {\em homogeneous} Dirichlet conditions for the velocity, {\em homogeneous} Neumann conditions for the constraint.
\end{itemize}

The linear operators are given by:\label{coeff-modif}
$$\begin{array}{ccl}
\widetilde{\mathcal L_{\varepsilon}^{(A)}}(\widetilde{\cR})&=&\mathcal L_{\varepsilon}^{(A)}(\widetilde{\cR})+\Re\, \varepsilon^N\, (\widetilde{\cR}\cdot\nabla\cR_{\mathrm{lift}}+ \cR_{\mathrm{lift}}\cdot\nabla\widetilde{\cR}), \\
\widetilde{\mathcal L_{\varepsilon}^{(B)}}(\widetilde{\cR})&=&\mathcal L_{\varepsilon}^{(B)}(\widetilde{\cR})+\We\, \varepsilon^N (\widetilde{\cR}\cdot\nabla\cS_{\mathrm{lift}}+g_{a}(\nabla\widetilde{\cR}, \cS_{\mathrm{lift}})), \\
\widetilde{\mathcal L_{\varepsilon}^{(C)}}(\widetilde{\cS})&=&\mathcal L_{\varepsilon}^{(C)}(\widetilde{\cS})+\We\, \varepsilon^N (\cR_{\mathrm{lift}}\cdot\nabla\widetilde{\cS}+g_{a}(\nabla{\cR}_{\mathrm{lift}}, \widetilde{\cS})).
\end{array}
$$
The source terms are given by:
$$\begin{array}{rcl}
\widetilde{\mathcal F_\varepsilon} &=& \mathcal F_\varepsilon -\Re\, \left(\varepsilon^N \cR_{\mathrm{lift}}\cdot\nabla\cR_{\mathrm{lift}}+\mathcal L_{\varepsilon}^{(A)}(\cR_{\mathrm{lift}})\right)+ (1-r)\Delta \cR_{\mathrm{lift}}, \\ 
\widetilde{\mathcal G_\varepsilon} &=& \mathcal G_\varepsilon - \cS_{\mathrm{lift}} + \Di \Delta \cS_{\mathrm{lift}} - \We(\mathcal L_{\varepsilon}^{(B)}( \cR_{\mathrm{lift}})+ \mathcal L_{\varepsilon}^{(C)}( \cS_{\mathrm{lift}}))\\[0.1cm]
&&\hspace{3cm} -\varepsilon^N\, \We\, \left(\cR_{\mathrm{lift}} \cdot \nabla\cS_{\mathrm{lift}} + g_{a}(\nabla \cR_{\mathrm{lift}}, \cS_{\mathrm{lift}})\right).
\end{array}
$$

\begin{remark}\label{rem:estimateST}
From the Propositions~\ref{prop:estimateST} and~\ref{prop:estime-lift} we can obtain a bound on the new source terms:
$$
\|\widetilde{\mathcal F_\varepsilon}\|_{L^2} \leq C \varepsilon^{-1},\qquad  \|\widetilde{\mathcal G_\varepsilon} \|_{L^2} \leq C \varepsilon^{-1}.
$$
\end{remark}

\begin{theorem}\label{thm:liftremainder} The remainder satisfies:
$$
\|\nabla {\cR}\|_{L^2}+ \| {\cS}\|_{H^1} \leq C\, \varepsilon^{-1}.
$$
\end{theorem}

\begin{proof}
Due to the relations between $(\cR,\cS,\cQ)$ and $(\widetilde\cR,\widetilde\cS,\widetilde\cQ)$
$$
\widetilde{\cR}:={\cR}-{\cR}_{\mathrm{lift}},\quad 
\widetilde{\cQ}:=Q, \quad  
\widetilde{\cS}:={\cS}-{\cS}_{\mathrm{lift}},
$$
using the Proposition~\ref{prop:estime-lift}, it suffices to analyze the error on~$(\widetilde\cR,\widetilde\cS,\widetilde\cQ)$.

The estimate is then obtained using a classical energy estimate on the system satisfied by~$(\widetilde\cR,\widetilde\cS,\widetilde\cQ)$.
More precisely, we first take the scalar product in~$L^2(\omega_\varepsilon)$ of the momentum equation by $2r\, \widetilde\cR$. Next we take the scalar product in~$L^2(\omega_\varepsilon)$ of the constitutive equation by~$\widetilde\cS$. We finally add the results to obtain
\begin{equation}\label{estimate0}
2r(1-r) \int_\omega \| \nabla \widetilde{\cR} \|^2 + \int_\omega \| \widetilde\cS \|^2 + \Di \int_\omega \|\nabla \widetilde\cS \|^2
= \text{RHS}.
\end{equation}
The term RHS is composed as follow:
$$
\begin{aligned}
RHS = &
- 2r \int_{\omega_\varepsilon} \widetilde{\mathcal L_{\varepsilon}^{(A)}}(\widetilde{\cR}) \cdot \widetilde{\cR}
+ r \int_{\omega_\varepsilon} \widetilde{\mathcal F_\varepsilon} \cdot \widetilde{\cR}
- \We\, \varepsilon^N \int_\omega g_a(\nabla \widetilde{\cR} , \widetilde\cS) : \widetilde\cS \\
& - \int_{\omega_\varepsilon} \widetilde{\mathcal L_{\varepsilon}^{(B)}}(\widetilde{\cR}) \cdot \widetilde{\cS}
- \int_{\omega_\varepsilon} \widetilde{\mathcal L_{\varepsilon}^{(C)}}(\widetilde{\cS}) \cdot \widetilde{\cS}
+ \int_{\omega_\varepsilon} \widetilde{\mathcal G_\varepsilon} \cdot \widetilde{\cS}.
\end{aligned}
$$
It is not very difficult to show that the source terms and linear terms of RHS can be controlled by the terms on the left-hand side of the estimate~\eqref{estimate0}. However the quadratic term does not lead in a straightforward way to a suitable estimate. Therefore we have to consider a new argument which is based on a fixed point procedure. Let us consider the following linearized system, denoted $(\mathrm R_{\mathrm{lin.}})$:
$$\begin{array}{rcl}
\varepsilon^N\, \Re\, \widetilde{\cR}^n \cdot \nabla \widetilde{\cR}^{n+1} + \widetilde{\mathcal L_{\varepsilon}^{(A)}}(\widetilde{\cR}^{n+1})- (1-r)\Delta \widetilde{\cR}^{n+1} + \nabla \widetilde{Q}^{n+1} \\
= \div(\widetilde{\cS}^{n+1})+ \widetilde{\mathcal F_\varepsilon},
\end{array}
$$
$$\begin{array}{r}
\varepsilon^N\, \We\, \left( \widetilde{\cR}^n \cdot \nabla \widetilde{\cS}^{n+1} + g_a(\nabla \widetilde{\cR}^n, \widetilde{\cS}^{n+1}) \right) + \widetilde{\mathcal L_{\varepsilon}^{(B)}}(\widetilde{\cR}^{n+1}) + \widetilde{\mathcal L_{\varepsilon}^{(C)}}(\widetilde{\cS}^{n+1})\\ 
+ \widetilde{\cS}^{n+1} - \Di \Delta \widetilde{\cS}^{n+1} =2r \D(\widetilde{\cR}^{n+1})+ \widetilde{\mathcal G_\varepsilon}, 
\end{array}
$$
and
$$\begin{array}{rcl}
\div \widetilde{\cR}^{n+1}&=&0.
\end{array}
$$
where $(\widetilde{\cR}^n,\widetilde{\cS}^n)$ are given. The idea relies on the following arguments:
\begin{enumerate}
\item we show that $(\widetilde{\cR}^n, \widetilde{\cS}^n)_{n}$ is bounded in $H^1(\omega_{\varepsilon})$, up to smallness assumptions~;
\item by the Cauchy criterion, we show that the sequence $(\widetilde{\cR}^n, \widetilde{\cS}^n)_{n}$ converges in $H^1(\omega_{\varepsilon})$~;
\item we let $n$ tend to $+\infty$ and show that the limit of $(\widetilde{\cR}^n, \widetilde{\cS}^n)_{n}$ is the solution of the system satisfied by the remainder. The limit still satisfies the estimates of step~1.
\end{enumerate}

{\bf Step 1.} Using the classical energy estimate, we have
\begin{equation*}
2r(1-r) \int_\omega \| \nabla \widetilde{\cR}^{n+1} \|^2 + \int_\omega \| \widetilde\cS^{n+1} \|^2 + \Di \int_\omega \|\nabla \widetilde\cS^{n+1} \|^2
= \text{RHS}^{(n)}.
\end{equation*}
The term $RHS^{(n)}$ is composed as follow:
$$
\begin{aligned}
RHS^{(n)} = &
- 2r \int_{\omega_\varepsilon} \widetilde{\mathcal L_{\varepsilon}^{(A)}}(\widetilde{\cR}^{n+1}) \cdot \widetilde{\cR}^{n+1}
+ r \int_{\omega_\varepsilon} \widetilde{\mathcal F_\varepsilon} \cdot \widetilde{\cR}^{n+1}\\
& -\displaystyle \We\, \varepsilon^N \int_\omega g_a(\nabla \widetilde{\cR}^n , \widetilde\cS^{n+1}) : \widetilde\cS^{n+1} \\
& - \int_{\omega_\varepsilon} \widetilde{\mathcal L_{\varepsilon}^{(B)}}(\widetilde{\cR}^{n+1}) \cdot \widetilde{\cS}^{n+1}
- \int_{\omega_\varepsilon} \widetilde{\mathcal L_{\varepsilon}^{(C)}}(\widetilde{\cS}^{n+1}) \cdot \widetilde{\cS}^{n+1}\\
& + \int_{\omega_\varepsilon} \widetilde{\mathcal G_\varepsilon} \cdot \widetilde{\cS}^{n+1}.
\end{aligned}
$$
We distinguish three types of terms:
\begin{itemize}
\item {\em source terms} $$ r \int_{\omega_\varepsilon} \widetilde{\mathcal F_\varepsilon} \cdot \widetilde{\cR}^{n+1}+ \int_{\omega_\varepsilon} \widetilde{\mathcal G_\varepsilon} \cdot \widetilde{\cS}^{n+1}$$
Using Cauchy-Schwarz, Poincar\'e and Young inequalities, we have
$$\begin{array}{rcl}
\left| r \displaystyle\int_{\omega_\varepsilon} \widetilde{\mathcal F_\varepsilon} \cdot \widetilde{\cR}^{n+1}\right| &\leq& r \| \widetilde{\mathcal F_\varepsilon}\|_{L^2(\omega_{\varepsilon})} \| \widetilde{\cR}^{n+1}\|_{L^2(\omega_{\varepsilon})}\\
&\leq& r C_{P} \| \widetilde{\mathcal F_\varepsilon}\|_{L^2(\omega_{\varepsilon})} \| \nabla \widetilde{\cR}^{n+1}\|_{L^2(\omega_{\varepsilon})}\\ 
&\leq& r(1-r) \| \nabla \widetilde{\cR}^{n+1}\|_{L^2(\omega_{\varepsilon})}^2+\displaystyle\frac{r}{4(1-r)}C_{P}^2 \| \widetilde{\mathcal F_\varepsilon}\|_{L^2(\omega_{\varepsilon})}^2
\end{array}$$
As a matter of fact, $r(1-r) \| \nabla \widetilde{\cR}^{n+1}\|_{L^2(\omega_{\varepsilon})}^2$ can be absorbed by the heft-hand side of the energy estimate. The other source term can be treated in a very similar way. 

\item {\em linear terms} $$- 2r \int_{\omega_\varepsilon} \widetilde{\mathcal L_{\varepsilon}^{(A)}}(\widetilde{\cR}^{n+1}) \cdot \widetilde{\cR}^{n+1}- \int_{\omega_\varepsilon} \widetilde{\mathcal L_{\varepsilon}^{(B)}}(\widetilde{\cR}^{n+1}) \cdot \widetilde{\cS}^{n+1}
- \int_{\omega_\varepsilon} \widetilde{\mathcal L_{\varepsilon}^{(C)}}(\widetilde{\cS}^{n+1}) \cdot \widetilde{\cS}^{n+1}.$$
Conventional arguments are the H\"older inequality, the Sobolev injections like $H^1(\omega_\varepsilon) \subset L^4(\omega_\varepsilon)$, with constant denoted~$C_S$, and Young's inequality. For instance, the first term of~$\widetilde{\mathcal L_{\varepsilon}^{(A)}}(\widetilde{\cR}^{n+1})$, that is~$\widetilde{\cR}^{n+1} \cdot \nabla u_0$, can be treated as follows:
$$
\begin{aligned}
\Big| 2r \int_{\omega_\varepsilon} ( \widetilde{\cR}^{n+1} \cdot \nabla u_0 ) \cdot \widetilde{\cR}^{n+1} \Big|
& \leq 2r \, \|\widetilde{\cR}^{n+1}\|_{L^4(\omega_\varepsilon)}^2  \|\nabla u_0\|_{L^2(\omega_\varepsilon)} \\
& \leq 2r \, C_S \, \|\nabla \widetilde{\cR}^{n+1}\|_{L^2(\omega_\varepsilon)}^2  \|\nabla u_0\|_{L^2(\omega_\varepsilon)}.
\end{aligned}
$$
Under smallness assumptions on $u_{0}$ (and therefore on the data of the initial problem), the right-hand side of the above inequality can be absorbed by the left-hand side of the energy estimate. The other linear terms can be treated in a very similar way.

\item {\em quadratic terms}
$$
- \We\, \varepsilon^N \int_\omega g_a(\nabla \widetilde{\cR}^{n} , \widetilde\cS^{n+1}) : \widetilde\cS^{n+1}.
$$
Using the inequality
$$
\begin{array}{l}
\Big| \We\, \varepsilon^N \displaystyle\int_\omega g_a(\nabla \widetilde{\cR}^n , \widetilde\cS^{n+1}) : \widetilde\cS^{n+1} \Big| \\
\hspace{2cm} \leq \We\, \varepsilon^N \|\nabla \widetilde{\cR}^n\|_{L^2(\omega_\varepsilon)}  \|\widetilde{\cS}^{n+1}\|_{L^4(\omega_\varepsilon)}^2 \\
\hspace{2cm} \leq  \We\, \varepsilon^N \, C_S \|\nabla \widetilde{\cR}^n\|_{L^2(\omega_\varepsilon)}  \|\nabla \widetilde{\cS}^{n+1}\|_{L^2(\omega_\varepsilon)}^2\\
\end{array}
$$
for $\varepsilon$ sufficiently small (namely $\We\, \varepsilon^N \, C_S \|\nabla \widetilde{\cR}^n\|_{L^2(\omega_\varepsilon)} < \Di$), then the right-hand side of the above inequality can be absorbed by the left-hand side of the energy estimate.
\end{itemize}
From the above considerations, we deduce the following estimate
$$
 \| \nabla \widetilde{\cR}^{n+1} \|_{L^2(\omega_{\varepsilon})}^2 +  \| \widetilde{\cS}^{n+1} \|_{H^1(\omega_{\varepsilon})}^2\leq C  \|  \widetilde{\mathcal F_\varepsilon} \|_{L^2(\omega_{\varepsilon})}^2+C \|  \widetilde{\mathcal G_\varepsilon} \|_{L^2(\omega_{\varepsilon})}^2.
$$
By Remark~\ref{rem:estimateST}, we obtain
\begin{equation}\label{eq:esti1}
 \| \nabla \widetilde{\cR}^{n+1} \|_{L^2(\omega_{\varepsilon})}^2 +  \| \widetilde{\cS}^{n+1} \|_{H^1(\omega_{\varepsilon})}^2\leq C \varepsilon^{-2}.
\end{equation}
Finally, in order to get the induction step on $n$, it is sufficient to guarantee that $\We\, \varepsilon^{N-2} \, C_S C < \Di$ to get the uniform estimate on $( \widetilde{\cR}^{n} , \widetilde{\cS}^{n})$.
Note also that this condition is satisfied if~$\varepsilon$ is small enough. 

{\bf Step 2.}  We prove that $(\widetilde{\cR}^n, \widetilde{\cS}^n)_{n}$ is a Cauchy sequence in $H^1(\omega_{\varepsilon})$. Introducing
$$
\widehat{\cR}^{n+1}:=\widetilde{\cR}^{n+1}-\widetilde{\cR}^n,\qquad \widehat{\cS}^{n+1}:=\widetilde{\cS}^{n+1}-\widetilde{\cS}^n,\qquad \widehat{\cQ}^{n+1}:=\widetilde{\cQ}^{n+1}-\widetilde{\cQ}^n,
$$
we get by subtraction in $(\mathrm R_{\mathrm{lin.}})$
$$\begin{array}{r}
\varepsilon^N\, \Re\, (\widetilde{\cR}^n \cdot \nabla \widehat{\cR}^{n+1} +\widehat{\cR}^n \cdot \nabla \widetilde{\cR}^{n})+ \widetilde{\mathcal L_{\varepsilon}^{(A)}}(\widehat{\cR}^{n+1})- (1-r)\Delta \widehat{\cR}^{n+1} + \nabla \widehat{Q}^{n+1} \\
= \div(\widehat{\cS}^{n+1}),\\
\end{array}
$$
$$\begin{array}{r}
\varepsilon^N\, \We\, \left( \widetilde{\cR}^n \cdot \nabla \widehat{\cS}^{n+1}+\widehat{\cR}^n \cdot \nabla \widetilde{\cS}^{n} + g_a(\nabla \widetilde{\cR}^n, \widehat{\cS}^{n+1})+g_a(\nabla \widehat{\cR}^n, \widetilde{\cS}^{n}) \right) \\
+ \widetilde{\mathcal L_{\varepsilon}^{(B)}}(\widehat{\cR}^{n+1}) + \widetilde{\mathcal L_{\varepsilon}^{(C)}}(\widehat{\cS}^{n+1}) + \widehat{\cS}^{n+1} - \Di \Delta \widehat{\cS}^{n+1} = 2r \D(\widehat{\cR}^{n+1})
\end{array}
$$
and
$$\begin{array}{r}
\div \widehat{\cR}^{n+1}=0.
\end{array}
$$
Performing an energy estimate consists, again, in multiplying the first equation by $2r \widehat{\cR}^{n+1}$ then integrate over $\omega_{\varepsilon}$, multiplying the second equation by $\widehat{\cS}^{n+1}$ then integrate over $\omega_{\varepsilon}$ and sum up the two contributions. We use the following estimates:
\begin{itemize}
\item[$\bullet$] The first contributions is easily controlled, as $$\left|\varepsilon^N\, \Re\, \displaystyle\int_{\omega_{\varepsilon}}(\widetilde{\cR}^n \cdot \nabla \widehat{\cR}^{n+1})\cdot \widehat{\cR}^{n+1}\right|=0.$$
\item[$\bullet$] The second contribution satisfies:
$$\begin{array}{l}
\left|\varepsilon^N\, \Re\, \displaystyle\int_{\omega_{\varepsilon}}(\widehat{\cR}^n \cdot \nabla \widetilde{\cR}^{n}))\cdot \widehat{\cR}^{n+1}\right|\\
\hspace{1cm}\leq \varepsilon^N\, \Re  \| \widehat{\cR}^{n} \|_{L^4(\omega_{\varepsilon})}  \| \nabla \widetilde{\cR}^{n} \|_{L^2(\omega_{\varepsilon})}  \| \widehat{\cR}^{n+1} \|_{L^4(\omega_{\varepsilon})}\\
\hspace{1cm}\leq C \varepsilon^{N-1}  \| \nabla\widehat{\cR}^{n} \|_{L^2(\omega_{\varepsilon})} \| \nabla\widehat{\cR}^{n+1} \|_{L^2(\omega_{\varepsilon})}\\
\hspace{1cm}\leq r(1-r)  \| \nabla \widehat{\cR}^{n+1} \|_{L^2(\omega_{\varepsilon})}^2 + \frac{C}{4r(1-r)} \varepsilon^{2N-2}  \| \nabla\widehat{\cR}^{n} \|_{L^2(\omega_{\varepsilon})}^2,
\end{array}$$
where we have used the estimate established in Eq.~\eqref{eq:esti1} and then Sobolev imbedding and Young's inequality.
\item[$\bullet$] The third contribution satisfies
$$\left| \displaystyle\int_{\omega_{\varepsilon}} \widetilde{\mathcal L_{\varepsilon}^{(A)}}(\widehat{\cR}^{n+1})\cdot \widehat{\cR}^{n+1}\right| \leq  ||| \widetilde{\mathcal L_{\varepsilon}^{(A)}}|||\, \| \nabla\widehat{\cR}^{n+1} \|_{L^2(\omega_{\varepsilon})}^2,$$ 
where $||| \cdot |||$ denotes the operator norm from $L^2(\omega_{\varepsilon})$ to itself. Recalling the expression of $ \widetilde{\mathcal L_{\varepsilon}^{(A)}}$, it can be shown that $||| \widetilde{\mathcal L_{\varepsilon}^{(A)}}|||$ is arbitrarily small for sufficiently small data or $\varepsilon$.
\item[$\bullet$] The other contributions can be treated in a straightforward way or by arguments similar to the previous ones.
\end{itemize}
We thus deduce that, under smallness assumptions on the data and $\varepsilon$,
$$
 \| \nabla \widehat{\cR}^{n+1} \|_{L^2(\omega_{\varepsilon})}^2 +  \| \widehat{\cS}^{n+1} \|_{H^1(\omega_{\varepsilon})}^2\leq C \varepsilon^{2N-2} ( \| \nabla \widehat{\cR}^{n} \|_{L^2(\omega_{\varepsilon})}^2 +  \| \widehat{\cS}^{n} \|_{H^1(\omega_{\varepsilon})}^2).
$$ 
It means in particular that $(\widetilde{\cR}^{n+1}-\widetilde{\cR}^{n},\widetilde{\cS}^{n+1}-\widetilde{\cS}^{n})$ tends to $0$ as $n$ goes to $+\infty$. Consequently, $(\widetilde{\cR}^{n},\widetilde{\cS}^{n})$ is a Cauchy sequence in $H^1(\omega_{\varepsilon})$. The sequence converges to some $(\widetilde{\cR}^\star,\widetilde{\cS}^\star)$ in $H^1(\omega_{\varepsilon})$ which, by means of construction, satisfies the following estimate:
\begin{equation}\label{eq:esti1b}
 \| \nabla \widetilde{\cR}^{\star} \|_{L^2(\omega_{\varepsilon})}^2 +  \| \widetilde{\cS}^{\star} \|_{H^1(\omega_{\varepsilon})}^2\leq C \varepsilon^{-2}.
\end{equation}

{\bf Step 3.} Letting $n$ tend to $+\infty$ we come to the conclusion that $(\widetilde{\cR}^\star,\widetilde{\cS}^\star)$ is the unique solution of the system satisfied by the lifted remainder. Therefore, the lifted remainder satisfies:
$$
\|\nabla \widetilde{\cR}\|_{L^2}+ \| \widetilde{\cS}\|_{H^1} \leq C\, \varepsilon^{-1}
$$
which concludes the proof.
\end{proof}

\begin{corollary} The asymptotic expansion is valid at any order.\end{corollary}

\section{Algorithm and numerical procedure}\label{Section6}

We show in this section that it is possible to effectively evaluate all the contributions of the ansatz. The only difficulty is to prove that the solution of each problem can be built using the only previous elementary solutions. For instance, let us prove that $(u_k,p_k,\sigma_k)$, which is the solution of the problem~$pb^{(k)}$, only depends on $(u_j,p_j,\sigma_j)$, $j<k$ and on $(U_j,P_j,\Sigma_j)$, $j\leq k$.
This is not so obvious since the problem~$pb^{(k)}$ calls for the use of a parameter~$a_k$, which seems to be related to a forthcoming problem~$PB^{(k+2)}$ through the relationship
$$
a_k = \lim_{Y\to +\infty}\int_{\T^d} F_{k+2}(X,Y)\, \mathrm dX.
$$
However we know prove that the definition of~$a_k$ is a consistent:
\begin{proposition}\label{prop:coeffcomput1} Coefficients $a_{k}$ (see the definition of $\mathrm{pb}^{(k)}$ and $\mathrm{PB}^{(k+2)}$) satisfying Eq.~\eqref{eq:ak} can be written as
$$
\begin{aligned}
a_k = \lim_{Y\to +\infty}\int_{\T^d} \Big[  & - \Re \, u_0 \cdot \nabla_x U_k + (1-r) \Delta_x U_{k} \\
& -\Re \sum_{i=1}^{k-1} u_i \cdot \nabla_x U_{k-i}  - \Re \sum_{i=1}^k U_i\cdot ( \nabla_x U_{k-i} + \nabla u_{k-i}) \Big].
\end{aligned}
$$
\end{proposition}

\begin{proof}
The source contribution~$F_{k+2}$ can be written as $F_{k+2}^A + F_{k+2}^B$ with
\begin{equation*}
\begin{aligned}
F_{k+2}^A = & \Big[  - \Re \, u_0 \cdot \nabla_x U_k + (1-r) \Delta_x U_{k} + \div_x (\Sigma_{k}) \\
& -\Re \sum_{i=1}^{k-1} u_i\cdot ( \nabla U_{k+1-i} + \nabla_x U_{k-i} ) \\
& - \Re \sum_{i=1}^k U_i\cdot ( \nabla U_{k+1-i} + \nabla_x U_{k-i} + \nabla u_{k-i}) \Big], \\
F_{k+2}^B = & -\Re \, u_0\cdot \nabla U_{k+1} -\Re \, u_k\cdot \nabla U_1 + \div (\Sigma_{k+1}) \\
& \quad + 2(1-r) \nabla_x \cdot \nabla U_{k+1} - \nabla_x P_{k+1}.
\end{aligned}
\end{equation*}
The first contribution~$F_{k+2}^A$ using the already defined elementary solutions ($(u_j,p_j,\sigma_j)$, $j<k$ and $(U_j,P_j,\Sigma_j)$, $j\leq k$), whereas the second ones~$F_{k+2}^B$ using elementary solutions of problem~$pb^{(k)}$ and~$PB^{(k+1)}$.\\

We now prove that $\displaystyle \lim_{Y\to +\infty}\int_{\T^d} F_{k+2}^B = 0$. For instance, we treat the first contribution (the four other contributions are similarly treated)
\begin{equation*}
\begin{aligned}
u_0\cdot \nabla U_{k+1} = \sum_{\ell=1}^d u_0^{(\ell)} \partial_{X_\ell} U_{k+1} - u_0^{(d+1)} \partial_{Y} U_{k+1}.
\end{aligned}
\end{equation*}
Taking the $X$-average, using the periodicity we obtain
\begin{equation*}
\begin{aligned}
\int_{\T^d} u_0\cdot \nabla U_{k+1} = u_0^{(d+1)} \partial_{Y} \Big( \int_{\T^d} U_{k+1} \Big).
\end{aligned}
\end{equation*}
Due to the behavior of the mean value $\int_{\T^d} U_{k+1}$ (see the Proposition~\ref{prop:blp1}), we have
$$\lim_{Y\to +\infty}\int_{\T^d} u_0\cdot \nabla U_{k+1} = 0.$$
To conclude the proof, it suffices to note that some contributions of~$F_{k+2}^A$ vanish too.
\end{proof}

The problem~$pb^{(k)}$ calls for the use of a parameter~$b_k$, which seem to be related to a forthcoming problem~$PB^{(k+1)}$ through the relationship
$$
b_{k}=\int_{\T^d} N_{k+1}(X)\, \mathrm dX - \Di^{-1} \int_{Y<0} G_{k+1}(X,Y)\, \mathrm dX\, \mathrm dY.
$$
However we know prove that the definition of~$b_k$ is a consistent:
\begin{proposition}\label{prop:coeffcomput2}
Coefficients $b_k $ (see the definition of $\mathrm{pb}^{(k)}$ and $\mathrm{PB}^{(k+1)}$) satisfying Eq.~\eqref{eq:bk} only depend on the elementary solutions $(u_j,p_j,\sigma_j)_{j<k}$ and $(U_j,P_j,\Sigma_j)_{j\leq k}$.
\end{proposition}

\begin{proof}
The source contribution~$G_{k+1}$ only depends on the already defined elementary solutions ($(u_j,p_j,\sigma_j)$, $j<k$ and $(U_j,P_j,\Sigma_j)$, $j\leq k$).
Only one term in the boundary contribution~$N_{k+1}$ does not depend on these elementary solution: $N_{k+1}^B = (\nabla H \cdot \nabla_x) \sigma_{k-1}|_{\gamma^-_0}$.
Nevertheless, its $X$-average is equals to~$0$.
\end{proof}

The resulting application of Propositions \ref{prop:coeffcomput1} and \ref{prop:coeffcomput2} leads to the following algorithm which allows us to build the approximation of the solution at any fixed order of precision:

\begin{tabular}{|l}
 {\em Initialization:}\\
\hspace{0.5cm} \begin{tabular}{|l}
Compute $a_0$ (compatibility of problem~$\mathrm{PB}^{(2)}$ at the top)\\
Compute $b_0$ (compatibility of problem~$\mathrm{PB}^{(1)}$ at the bottom)\\
Compute $(u_0,p_0,\sigma_0)$\\
Compute $(U_1,P_1,\Sigma_1)$\\
\end{tabular}\\
{\em Iterative process on $k$:}\\
FOR $k=1,...,+\infty$, DO\\
\hspace{0.5cm} \begin{tabular}{|l}
Compute $a_k$ (compatibility of problem~$\mathrm{PB}^{(k+2)}$ at the top)\\
Compute $b_k$ (compatibility of problem~$\mathrm{PB}^{(k+1)}$ at the bottom)\\
Compute $(u_k,p_k,\sigma_k)$\\
Compute $(U_{k+1},P_{k+1},\Sigma_{k+1})$\\
\end{tabular}\\
END
\end{tabular}

\section{Concluding remarks}

\subsection{Boundary conditions} The diffusive Oldroyd model is generally associated to Neumann boundary conditions for the elastic stress tensor \cite{ORL00}, as suggested by the interpretation of the stress diffusion term as arising from the diffusion of polymeric dumbbells \cite{EKL89}. However some authors also considered Dirichlet boundary conditions \cite{RMC06} or mixed boundary conditions \cite{AOF07}. We emphasize that the method that has been developped in this paper readily adapts to the consideration of Dirichlet boundary conditions: the definition of the elementary problems has to be adapted and, in particular, the behaviour at infinity of the boundary layer elastic tensor corrector is completely determined by an exponential decay towards a constant which can be counter-balanced by using a suitable elementary problem at next order.

\subsection{Influence of the curvature of the channel} 

The analysis of the roughness effects has been led in a particular geometry: a channel. When considering space-varying profiles of the boundary such as nozzles or more general converging-diverging profiles, additional coupling effects have to be taken into account. Although the methodology developped in this paper still applies, source terms have to be added in the elementary problems in order to include the scale effects at the macroscopic scale onto the boundary term which serves as a correction in the boundary layer problems. However, there is no additional difficulty from the mathematical point of view, although it tends to increase the complexity of the description of the elementary problems.

\subsection{The non-diffusive model}

In many studies the standard Oldroyd model is considered without diffusion of the elastic stress tensor. The link between the standard model and the diffusive model has been investigated from the numerical point of view by the authors \cite{CM14} by considering the vanishing diffusion process in the diffusive model. In the context of the roughness issue, the adaptation of our analysis is questionable when considering the standard Oldroyd model only. Although the two models are very close (at least formally in the regime $\Di \rightarrow 0$), the asymptotic expansion proposed in this paper does not apply to the case $\Di=0$, even formally. This is due to the degeneracy of the boundary conditions: first, the loss of the boundary conditions in the standard model prevents us from developping the current strategy which is based upon the correction of the boundary terms at higher orders~; second, ignoring the specific treatment of the boundary conditions that was done in the diffusive model provides inconsistent elementary problems. 

\bibliographystyle{plain}
\bibliography{bibliography}

\end{document}